\documentclass[letterpaper,12pt]{amsart}

\usepackage{color,amscd}
\usepackage[draft]{changes}
\usepackage{enumerate}

\newcommand{\pa}{\partial}

\usepackage{mathtools}

\usepackage[abbrev]{amsrefs}
\usepackage{amssymb,amsmath,amsthm,xfrac}

\newtheorem{thm}{Theorem}[section]
\newtheorem{prop}[thm]{Proposition}

\newtheorem{Def}[thm]{Definition}
\newtheorem{lem}[thm]{Lemma}

\newtheorem{remark}[thm]{Remark}
\newtheorem{corl}[thm]{Corollary}
\newtheorem{example}[thm]{Example}

\newcommand{\Nat}{\mathbb{N}}

\newcommand{\eps}{\varepsilon}
\newcommand{\Hilbert}{\mathcal{H}}
\newcommand{\Dom}{\mathfrak{Dom}}

\def \<{\langle}
\def \>{\rangle}

\def \R{\mathbb R}

\def \H{{\cal H}}

\def \H^0{{\cal H}^0 or}

\def \w{\omega}

\def \p{\partial}
\def \n{\nabla}
\def \beq{\begin{equation}}
\def \eeq{\end{equation}}

\def \n{\nabla}

\def \eref{\eqref}

\begin{document}



\title{The spectrum of the Laplacian on forms}

\author{Nelia Charalambous}
\address{Department of Mathematics and Statistics, University of Cyprus, Nicosia, 1678, Cyprus} \email[Nelia Charalambous]{nelia@ucy.ac.cy}

 \author{Zhiqin Lu} \address{Department of
Mathematics, University of California,
Irvine, Irvine, CA 92697, USA} \email[Zhiqin Lu]{zlu@uci.edu}

\thanks{
The first author was partially supported by a University of Cyprus Start-Up grant. The second author is partially supported by the DMS-1510232.}
 \date{\today}



\begin{abstract}
In this article we prove a generalization of Weyl's criterion for the spectrum of a self-adjoint nonnegative operator on a Hilbert space. We will apply this new criterion in combination with Cheeger-Fukaya-Gromov and Cheeger-Colding theory  to study the  $k$-form essential spectrum over a complete manifold with vanishing curvature at infinity or asymptotically nonnegative Ricci curvature.

In addition, we will apply the generalized Weyl criterion to study the variation of the spectrum of a self-adjoint operator under continuous perturbations of the operator.  In the particular case of the Laplacian on $k$-forms over a complete  manifold we will use these analytic tools to find significantly stronger results for its spectrum  including its behavior under a continuous deformation of the metric of the manifold.
\end{abstract}
\maketitle
\section{Introduction}

Let $H$ be a densely defined, self-adjoint and nonnegative operator on a Hilbert space $\mathcal{H}.$ The spectrum of $H$, which we will denote $\sigma(H)$,  consists of all points $\lambda\in \mathbb{C}$ for which $H-\lambda I$ fails to be invertible. Since $H$ is nonnegative, its  spectrum  is contained in $[0,\infty)$. The  essential spectrum of  $H$, $\sigma_\textup{ess}(H)$, consists of the cluster points in the spectrum and of isolated eigenvalues of  infinite multiplicity. It is well known that both $\sigma(H)$ and $\sigma_{\rm ess}(H)$ are closed sets in  $\mathbb R$ and in $\mathbb{C}$.

In this article, we study the spectrum and the essential spectrum of the Hodge Laplacian on differential forms over a complete Riemannian manifold.  The set of points in the spectrum that do not belong to the essential spectrum are called the pure point spectrum. When the manifold is compact, the essential spectrum is an empty set, and the spectrum coincides with the pure point spectrum. When the manifold is complete and noncompact, both may exist. When the essential spectrum is a connected set, that is, it consists of an interval $[\alpha,\infty)$, then the only invariant is its lower bound $\alpha$.~\footnote{When the essential spectrum is the empty set, we define its lower bound to be $\infty$.} In this case, we say that the essential spectrum is \emph{computable}.
Unlike the pure point spectrum  which we can only compute in a few specialized cases, the essential spectrum is computable over a larger variety of manifolds.

We will find a large class of manifolds so that the $k$-form essential spectrum is computable. On the other hand, there exist examples of manifolds where the function essential spectrum has gaps (cf.~\cites{Post1,Lott2,SchoTr}). We believe that similar examples should exist for the $k$-form essential spectrum as well. Given that the presence of gaps has only been demonstrated in cases of manifolds with negative curvature, it seems that our results provide very broad curvature conditions so that a manifold has computable $k$-form essential spectrum.

The computation of the spectrum (essential spectrum) of the Laplacian requires the construction of a large class of test differential forms. This is a difficult task on a general manifold, since there exists only a small collection of canonically defined differential forms to work with. We will combine two somewhat ``opposite'' methods to achieve this goal: first, we introduce a new version of the Weyl criterion,  which greatly reduces the regularity and smoothness of the test differential forms; second, we make use of Cheeger-Fukaya-Gromov theory and Cheeger-Colding theory to obtain a new type of test differential forms on the manifold.

Although a preliminary version has been used in~\cite{char-lu-1}, the new Weyl criterion (Theorem \ref{Thm.Weyl.bis-4}) we provide is  more general and powerful, and is the one we will apply throughout this paper.  Moreover, our methodical use of collapsing theory for locating the essential spectrum seems to be new, and is an important component of the computation of the essential spectrum on forms.

The first  main result of this paper is the following
\begin{thm}\label{thm1}
Let $(M^n,g)$ be a complete noncompact Riemannian manifold which is asymptotically flat and whose cone at infinity is smooth except possibly at a single singularity point.  Then, the essential spectrum of the Laplacian on $k$-forms is either empty or is equal to $[\alpha_k,\infty)$ for a nonnegative number $\alpha_k$.
\end{thm}

We will show that whenever a manifold has Ricci curvature asymptotically nonnegative (which includes the case of asymptotical flatness), the bottom of its essential spectrum is captured by a sequence of large geodesic balls which we will call a minimal sequence (see Definition \ref{defMin}). Depending on the behavior of the injectivity radius of these geodesic balls, we will be able to further classify the essential spectrum on $k$-forms. Theorem \ref{thm1} can be inferred from the following stronger result.
\begin{thm} \label{mainthm}
Let $(M^n,g)$ be a complete noncompact asymptotically flat Riemannian manifold of dimension $n$.

(i) If there exists a minimal sequence of geodesic balls $B_{x_{i}}(R_{i})$ such that the injectivity radius at $x_i$, $I(x_i)$, satisfies $I(x_i) \to \infty,$ then   $\sigma_\mathrm{ess}(k,\Delta)=[0,\infty)$ for all $0\leq k \leq n$.

(ii) If there exists a minimal sequence of geodesic balls $B_{x_{i}}(R_{i})$ with  $0< a\leq I(x_i)  \leq b,$ then   $\sigma_\mathrm{ess}(k,\Delta)=[\alpha_k,\infty),$ for some constant $\alpha_k\geq 0$.

(iii) If there exists a minimal sequence of geodesic balls $B_{x_{i}}(R_{i})$ with $I(x_i)   \to 0$ and $B_{x_{i}}(R_{i})$ has a cone at infinity with at most a single singularity point, then either    $\sigma_\mathrm{ess}(k,\Delta)$ is empty, or  $\sigma_\mathrm{ess}(k,\Delta)=[\alpha_k,\infty)$ for some constant $\alpha_k\geq 0$.
\end{thm}
Note that the manifold could have varying types of minimal sequences as described in the theorem above. However, they will all lead to the same value for $\alpha_k$.  Theorem \ref{mainthm} is more general than Theorem \ref{thm1} since it only requires knowing the structure at infinity of the minimal geodesic sequence and not of the whole manifold. For a precise definition of a cone at infinity see Definition \ref{cone}.  The proof of Theorem \ref{mainthm} combines Cheeger-Fukaya-Gromov theory~\cite{CFG} and the new Weyl criterion. We will be addressing the noncollapsing case in Section~\ref{7} and the collapsing cases in Sections~\ref{8} and~\ref{9}. In Section \ref{9} we will also study a more complex case of a singular cone at infinity.

Here  and for the rest of the paper, we shall use either $\sigma(k,\Delta,M)$, or $\sigma(k,\Delta)$ ($\sigma_{\rm ess}(k,\Delta,M)$, or  $\sigma_{\rm ess}(k,\Delta)$ \emph{resp}.) to denote the spectrum (essential spectrum \emph{resp}.) of the Laplacian on $k$-forms, where $0\leq k\leq n$. We remark that since our assumption only concerns the ends of the manifold, we have no control over the point spectrum.

The behavior of the spectrum becomes more complicated when one only assumes Ricci curvature asymptotically nonnegative. Although we still get pointed convergence of a sequence of balls to a metric space in the (pointed) Gromov-Hausdorff sense, the limit is significantly singular.  In Section \ref{10} we will study the $k$-form spectrum over noncompact manifolds with Ricci curvature asymptotically nonnegative on a sequence of expanding balls. We will show
\begin{thm} \label{thmRic}
Let $(M^n,g)$  be a noncompact  complete  Riemannian manifold. Assume that on a sequence of disjoint geodesic balls $M_i=B_{x_i}(R_i)$ with $R_i\to \infty$, there exist $\delta_i \to 0$ such that
\[
\mathrm{Ric}_{M_i} \geq - \delta_i.
\]
Then, there exists a positive integer, $q>0$, such that for each $k\leq q$  and $k\geq n-q$
\[
\sigma(k,\Delta, M)=\sigma_{\rm ess}(k,\Delta, M) = [0,\infty).
\]
\end{thm}

The above result is also true when the manifold has Ricci curvature asymptotically nonnegative. The integer $q$ depends on the behavior of the manifold at infinity.  For example, if $M$ has nonnegative Ricci curvature and Euclidean volume growth, then $q=n$, and hence we know the $k$-form spectrum for all $k$. We remark that although in this case the manifold is a metric cone at infinity by Cheeger-Colding theory,  the classical Weyl criterion is not sufficient to compute the essential spectrum of the Laplacian on differential forms. Thus even in this special case, our result follows through a combination of our generalized Weyl criterion with Cheeger-Colding theory.

Theorem~\ref{thm1} and Theorem~\ref{mainthm} are obviously true for the covariant Laplacian as well.  Our result in Theorem~\ref{thmRic} is  true for the covariant Laplacian, if we assume that the manifold is asymptotically Ricci flat. We remark that  in this case, even though the curvature tensor may blow up at infinity, our spectrum results are true for both the Hodge and the covariant Laplacian.

We will apply the generalized Weyl criterion to obtain other interesting results for the $k$-form spectrum. We  show that a point in the $k$-form spectrum of the Laplacian  must belong to either the $(k-1)$-form spectrum or to the  $(k+1)$-form spectrum  (Theorem \ref{thm31}). It then immediately follows that the   spectrum of the Laplacian on 1-forms always contains the  spectrum of the Laplacian on functions (Corollary \ref{cor31}). We emphasize the fact that this theorem is an analytic result, which does not impose any assumptions on the curvature nor the volume growth of the manifold. Our results  show that we do not always have to make stronger geometric assumptions on the manifold to compute the $k$-form spectrum, in comparison to the Laplacian on functions. In Section \ref{SecForm} we will join Theorem \ref{thm31} with a previous result of ours to conclude that on manifolds with Ricci curvature asymptotically nonnegative in the radial direction, the  essential spectrum of the Laplacian on 1-forms is $[0,\infty)$, whenever the volume of the manifold is infinite, or if the volume is finite but the volume does not decay exponentially.

We will also show that the spectrum of the Laplacian on $k$-forms over a noncompact manifold varies continuously with the metric (Theorem \ref{thmCont}). Our generalized Weyl criterion will allow us to relate the spectrum of the Laplacian to the spectra of $\delta d$ and $d \delta$ (Lemma \ref{lemSp}). We  show that the spectrum of these partial operators varies continuously with the metric, and then use this result together with Lemma \ref{lemSp} to prove the continuous variation of the spectrum of the full Laplacian. Theorem \ref{thmCont} generalizes the work of Dodziuk who considered the compact case \cite{dod}. Our result is also related to the spectral continuity  results due to Fukaya~\cite{Fuk3} and Cheeger-Colding~\cite{CCoIII} for the function spectrum and to those of Lott~\cites{lott, lott1}, and more recently Honda~\cite{Ho}, on the form spectrum.

For our applications, it would be convenient to consider the spectrum of any nonnegative self-adjoint operator $H$ as a complete metric space. As is well-known, the spectrum $\sigma(H)$  is a closed subset of $[0,\infty)$, hence it is a complete metric space with distance function the one induced from $\mathbb R$. Let
\[
\tilde \sigma(H)=\{-1\}\cup \sigma(H).
\]
Then $(\tilde \sigma(H),-1)$ is a pointed complete metric space. By the Gromov compactness theorem, the set of all spectra is a pre-compact set\footnote{{In fact, it is a compact set.}} under the pointed Gromov-Hausdorff distance, which we denote as $d_{GH}$.

Similar notions can be defined for the essential spectrum. We feel that the above definition is a convenient notion in many of our results, and we shall use it throughout this paper, for example in Sections \ref{S3} and \ref{S5}. However, the point $-1$ is just an abstract point, and in order to simplify notation, we shall use
$\sigma(H)$, instead of the more complicated $(\tilde \sigma(H),-1)$, for the rest of the paper. \\

{\bf Acknowledgement.} The authors would like to thank  Kenji Fukaya, John Lott,  Rafe Mazzeo, and Xiaochun Rong  for their useful comments and suggestions.

\section{A Generalized Weyl Criterion}
%
Let~$H$ be a  densely defined,  self-adjoint and nonnegative operator on a Hilbert space~$\Hilbert$.
The norm and inner product on~$\Hilbert$ are respectively
denoted by~$\|\cdot\|$ and $(\cdot,\cdot)$. Let $\Dom(H)$ denote the domain of $H$.
The classical Weyl criterion states that
\begin{thm}[Classical Weyl criterion]\label{Thm.Weyl}
A complex number $\lambda$ belongs to $\sigma(H)$ if, and only if,
there exists a sequence $\{\psi_j\}_{j \in \Nat} \subset \Dom(H)$
such that
\begin{enumerate}
\item
$
  \forall \,j\in\Nat, \quad
  \|\psi_j\|=1
$\,,
\item
$
  (H-\lambda)\psi_j \to 0, \text{ as } j\to\infty  \text{ in }\mathcal H.$
\end{enumerate}
Moreover, $\lambda$ belongs to  $\sigma_\mathrm{ess}(H)$ if, and only if,
in addition to the above properties
\begin{enumerate}
\setcounter{enumi}{2}
\item
$
  \psi_j \to  0 \text{ weakly as  } j\to\infty \text{ in }\mathcal H.
$
\end{enumerate}
\end{thm}

In this  section we  prove a similar  but  both qualitatively and quantitatively stronger criterion for the spectrum and the essential spectrum of $H$.  Let $f, g$ be two  bounded positive continuous functions on $[0,\infty)$, with $g$ satisfying the following additional property: for any $\lambda\geq 0$,  there exists a positive  constant $c$ such that $g(t)(t-\lambda)\geq c>0$ on the interval $[\lambda+1,\infty)$.

We define the following three constants: for a fixed $\lambda\geq 0$ we set
\begin{align}\label{constants}
\begin{split}
& c_1(\lambda)=\inf_{t\in[0,\lambda]} f(t);\\
& c_2(\lambda)=\min\left(\inf_{t\in[\lambda,\lambda+1]} g(t), \inf_{t\in[\lambda+1,\infty)} g(t) (t-\lambda)\right);\\[0.5em]
& c_3(\lambda)=\lambda\,\sup_{t\in[0,\lambda]} g(t).
\end{split}
\end{align}

Moreover, let
\[
c_0=\max(\sup f(t), \sup g(t)).
\]

\begin{thm}\label{Thm.Weyl.bis-4}
Let $H$ be defined as above. We fix a nonnegative  number  $\lambda\geq 0$, and a small positive number $0<\delta<c_0$. If
\[
{\rm dist}(\lambda,\sigma(H))<\delta/c_0,
\]
then there exists a sequence $\{\psi_j\}_{j \in \Nat} \subset  \Dom(H)$ such that\footnote{the $\psi_j$ could all be identical.}
\begin{enumerate}
\item
$
  \forall \, j\in\Nat, \quad
  \|\psi_j\|=1
$\,,
\item
$
 |(f(H) (H-\lambda)\psi_j, (H-\lambda)\psi_j)|\leq \delta,  \quad {and}
$
\item
$
|(g(H)\psi_j, (H-\lambda)\psi_j)|\leq \delta.
$

\end{enumerate}
Whenever
\[
{\rm dist}(\lambda,\sigma_\mathrm{ess}(H))<\delta/c_0,
\]
 then in addition to the above properties, we have
\begin{enumerate}
\setcounter{enumi}{3}
\item
$
  \psi_j \to 0, \text{ weakly as } j\to\infty
$
\text{ in } $\mathcal H$.

\end{enumerate}

On the other hand, if properties (1),(2),(3) are satisfied for a sequence of $\{\psi_j\}_{j \in \Nat} \subset \Dom(H)$, then
\[
{\rm dist}(\lambda,\sigma(H))\leq \left(\frac{c_1(\lambda)+c_2(\lambda)+c_3(\lambda)}{c_1(\lambda)\,c_2(\lambda)}\right)^{\frac 13}\cdot\delta^{\frac 13}
\]
in the case $\lambda>0$, and
\[
{\rm dist}(\lambda,\sigma(H))\leq  \frac{\delta}{ c_2(\lambda)}
\]
in the case $\lambda=0$.

If the $\{\psi_j\}_{j \in \Nat}$ also satisfy property (4), then the above upper bound also holds for ${\rm dist}(\lambda,\sigma_\mathrm{ess}(H))$.
\end{thm}

\begin{proof}  By the assumptions on $H$,
\begin{equation}\label{decomp}
H=\int_0^\infty t\, dE(t)
\end{equation}
for some spectral measure $E$. Assume that ${\rm dist}(\lambda,\sigma(H))<\delta/c_0$. Then there exists a sequence $\{\psi_j\}$ such that
\[
\|(H-\lambda)\psi_j\|<\delta/c_0, \quad \|\psi_j\|=1.
\]
Then we have
\[
(f(H) (H-\lambda)\psi_j, (H-\lambda)\psi_j)=\int_0^\infty f(t)(t-\lambda)^2 d\|E(t)\psi_j\|^2.
\]
As a result,
\[
|(f(H) (H-\lambda)\psi_j, (H-\lambda)\psi_j) |\leq c_0\int_0^\infty (t-\lambda)^2 d\|E(t)\psi_j\|^2=c_0\|(H-\lambda)\psi_j\|^2\leq\delta,
\]
and similarly,
\[
|(\,g(H) \psi_j, (H-\lambda)\psi_j)|\leq c_0\,\|\psi_j\|\cdot\|(H-\lambda)\psi_j\|\leq\delta.
\]

If in addition,
\[
{\rm dist}(\lambda,\sigma_\mathrm{ess}(H))<\delta/c_0,
\]
then we may assume that the sequence $~\{\psi_j\}$ is orthogonal (cf. \cite{Don81}*{Prop. 2.2}).  Since the $\psi_j$ are of unit norm and orthogonal, they are weakly convergent to $0$.

To prove the reverse statement of the theorem, we  first consider the case $\lambda>0$ and assume that ${\rm dist}(\lambda,\sigma(H))>\eps$  for some $\eps<\min(\lambda,1)$.  Let $P=E([0,\lambda-\eps])$. $P$ is a projection operator which we use to write
\[
\psi_j=\psi_j^1+\psi_j^2,
\]
where
\[
\psi_j^1=\int_0^{\lambda-\eps} dE(t)\psi_j=P \psi_j,
\]
and $\psi_j^2=\psi_j-\psi_j^1$. Then we have
\begin{align*}&
(f(H) (H-\lambda)\psi_j, (H-\lambda)\psi_j) \\&=(f(H) (H-\lambda)\psi_j^1, (H-\lambda)\psi_j^1)
+(f(H) (H-\lambda)\psi^2_j, (H-\lambda)\psi_j^2)\\
& \geq c_1(\lambda)\eps^2\|\psi_j^1\|^2+(f(H) (H-\lambda)\psi^2_j, (H-\lambda)\psi_j^2)\geq c_1(\lambda)\eps^2\|\psi_j^1\|^2,
\end{align*}
On the other hand, we similarly get
\[
(g(H)\psi_j, (H-\lambda)\psi_j)\geq c_2(\lambda)\eps\|\psi_j^2\|^2-c_3(\lambda)\|\psi_j^1\|^2.
\]
 If the criteria {\it (2)}, {\it (3)} are satisfied, then, by the two inequalities above, we conclude
 \begin{align*}
 & \delta\geq  c_1(\lambda)\eps^2\, x,\\
 &\delta\geq c_2(\lambda)\eps(1-x)-c_3(\lambda) x,
 \end{align*}
 where $x=\|\psi_j^1\|^2$. Thus we must have
 \[
 \delta\geq\frac{c_1(\lambda)\, c_2(\lambda)\,\eps^3}{c_1(\lambda)\eps^2+c_2(\lambda)\eps+c_3(\lambda)}\geq \frac{c_1(\lambda)\, c_2(\lambda)\,\eps^3}{c_1(\lambda)+c_2(\lambda)+c_3(\lambda)}.
 \]
 This proves the upper bound estimate of $\eps$. Moreover, if (4) is satisfied, then the estimate holds for  $\sigma_\mathrm{ess}(H)$.

In the case $\lambda=0$ we have that $c_3(\lambda)=0$ and $x=0$ in the above argument, as a result we get the estimate $\eps<\delta/c_2(\lambda)$.

\end{proof}

We shall take
$f(t)=(t+\alpha)^{-2}$ and $g(t)=(t+\alpha)^{-1}$ for some positive $\alpha$. Observe that
\begin{equation*}
\begin{split}
&((H+\alpha)^{-2} (H-\lambda)\psi_j, (H-\lambda)\psi_j) \\
&=((H+\alpha)^{-1}\psi_j, (H-\lambda)\psi_j)-(\alpha+\lambda)\;((H+\alpha)^{-2}\psi_j, (H-\lambda)\psi_j).
\end{split}
\end{equation*}
Then we have the following

\begin{corl}\label{cor21} A nonnegative real number  $\lambda$ belongs to the spectrum $\sigma(H)$ if, and only if, there exists a positive constant $\alpha$ and a sequence $\{\psi_j\}_{j \in \Nat} \subset \Dom(H)$ such that

\begin{enumerate}
\item
$
\forall j\in\Nat, \quad  \|\psi_j\|=1
$\,,

\item
$
((H+\alpha)^{-m}\psi_j, (H-\lambda)\psi_j)\to 0 \
$
for $m=1,2$.
\end{enumerate}

Moreover, $\lambda$ belongs to the essential spectrum $\sigma_\mathrm{ess}(H)$  if, and only if,
in addition to the above properties

\begin{enumerate}
\setcounter{enumi}{2}
\item $  \psi_j \to 0, \text{ weakly as } j\to\infty$\text{ in } $\mathcal H$.
\end{enumerate}
Furthermore, if for some $0<\delta<1$,
\[
|((H+\alpha)^{-m}\psi_j, (H-\lambda)\psi_j)|\leq\delta
\]
for both $m=1,2$ and all $j$, then there exists a constant $c(\lambda,\alpha)>0$, depending only on $\lambda,\alpha$, such that
\[
{\rm dist}(\lambda,\sigma(H))<c(\lambda,\alpha)\,\delta^{\frac 13}.
\]
\end{corl}

\vspace{.2in}

\section{Continuous Deformation of the Spectrum} \label{S3}

Let $\mathcal{H}$ be a Hilbert space with two inner products $(\cdot,\cdot)_0$ and $(\cdot,\cdot)_1$. Let $H_0, H_1$ be two  densely defined nonnegative operators on $\mathcal{H}$ that are self-adjoint with respect to the inner products $(\cdot,\cdot)_0$ and $(\cdot,\cdot)_1$ respectively. Let $Q_0, Q_1$ be their respective quadratic forms and denote the two norms on $\mathcal{H}$ by $\|\cdot\|_0$ and $\|\cdot\|_1$. Note that both $Q_0$ and $Q_1$ are  nonnegative.

We assume that there exists a dense subspace\footnote{$\mathcal C$ is the space of smooth forms  with  compact support in our applications.}  $\mathcal C\subset \mathcal H$ such that $\mathcal C$ is contained in $ \mathfrak {Dom} (H_0)\cap  \mathfrak {Dom} (H_1)$.

\begin{Def}\label{defBd}
We say that the operators $H_0, H_1$ are $\eps$-close, if there exists a positive constant $0<\eps<1$  such that  for all $u\in \mathcal C$ the following two inequalities hold
\begin{align}
(1-\eps)\,\|u\|_0^2 &\leq \|u\|_1^2  \leq (1+\eps)\,\|u\|_0^2;  \label{A1}\\
(1-\eps)\,Q_0(u,u) & \leq Q_1(u,u) \leq (1+\eps)\,Q_0(u,u) \label{A2}.
\end{align}
\end{Def}

If $H_0,H_1$ are $\eps$-close, then for any $u,v \in  \mathcal C$
\begin{align}
|(u,v)_1-(u,v)_0| & \leq \eps (\|u\|_0\, \|v\|_0); \label{Bd1}\\
|Q_1(u,v) - Q_0(u,v)| & \leq \eps \, [Q_0(u,u) \,  Q_0(v,v)]^{1/2}. \label{Bd2}
\end{align}

To prove \eref{Bd1} we first observe that if either $u=0$ or $v=0$, then clearly the inequality holds. So we will prove it when neither vanishes. We first observe that
\begin{equation*}
\begin{split}
|(u,v)_1-(u,v)_0| &= \frac 14 \; \bigl| \, [\|u+v\|_1^2 - \|u+v\|_0^2 - (\|u-v\|_1^2 - \|u-v\|_0^2)]\,\bigr| \\
& \leq \frac 14 \,\eps \, [ \|u+v\|_0^2 + \|u-v\|_0^2] \leq \frac 12 \, \eps\, [ \|u\|_0^2 + \|v\|_0^2].
\end{split}
\end{equation*}
If in the above inequality we replace $u$ by $a u$ and $v$ by $v/a$ with $a=\|v\|_0/\|u\|_0$, then \eref{Bd1} follows immediately. Inequality \eref{Bd2} follows in a similar manner given that $Q_0$ and $Q_1$ are nonnegative.

We will show that the spectra of two $\eps$-close operators  also remain close. Our proof will require a comparison of their resolvent operators which is Lemma \ref{resBd} below.  We first prove the boundedness of the resolvents on $\mathcal H$.
\begin{lem} \label{lem1b}  Let $H$ be defined as above. Then for any nonnegative integer $m$ and $\alpha>0$,  $(H+\alpha)^{-m}$ is a bounded operator on $\mathcal H$.
\end{lem}
\begin{proof}
By the spectral decomposition~\eqref{decomp} of $H$, we can write
\[
(H+\alpha)^{-m}=\int_0^\infty (\lambda+\alpha)^{-m} dE.
\]
Since $(\lambda+\alpha)^{-m}\leq \alpha^{-m}$, the operator is bounded,  and in fact its operator norm is at most $\alpha^{-m}$.
\end{proof}
\begin{lem} \label{resBd}
Let $H_0$ and $H_1$ be two self-adjoint nonnegative operators that are $\eps$-close on $\mathcal{H}$ as in Definition \eref{defBd}, with $0<\eps<1/2$.

Fix $\alpha>1$. Then  for all $u,v \in \mathcal  C$
\begin{equation*}
\bigl|(\, (H_1+\alpha)^{-m}u,v\,)_1 - (\, (H_0+\alpha)^{-m}u,v\,)_0 \bigr| \leq    (2m+1)\,\eps\, \|u\|_0 \|v\|_0
\end{equation*}
for any nonnegative integer $m\geq 0$.
\end{lem}
\begin{proof} For $m=0$, the result follows from ~\eqref{Bd1}.

Let $m>0$ and assume that the lemma holds for $m-1$. Then \begin{align*}
&|((H_1+\alpha)^{-m+1} u, (H_1+\alpha)^{-1}v)_1-
((H_0+\alpha)^{-m+1} u, (H_1+\alpha)^{-1}v)_0|\\
&\qquad\leq (2m-1)\,\eps
\| u\|_0 \| v\|_0.
\end{align*}
Let $w=(H_0+\alpha)^{-m+1} u$, $w_1=(H_0+\alpha)^{-1}w$, and $v_1=(H_1+\alpha)^{-1} v$. Then
\begin{align*}
&|(w,v_1)_0- (w_1,v)_0|\\
& \leq |((H_0+\alpha)w_1,v_1)_0-(w_1, (H_1+\alpha)v_1)_1|
+|(w_1,v)_0-(w_1,v)_1|\\
&
\leq \eps\, ((Q_0(w_1,w_1)\cdot Q_0(v_1,v_1))^{1/2}+\|w_1\|_0\cdot\|v\|_0)\leq 2\eps \|u\|_0\cdot\|v\|_0.
\end{align*}
The result follows for $m$ after combining the above two estimates.
\end{proof}

We will now describe the proximity of the spectra of two $\eps$-close operators.

\begin{thm} \label{PropBd}
Let $H_0, H_1$ be two nonnegative operators on $\mathcal{H}$ that are $\eps$-close as in Definition \ref{defBd} for some $0<\eps<1/2$.
Fix $A>0$. Then  for any  $\lambda \in \sigma(H_1)\cap [0,A]$
\[
{\rm dist}(\lambda,\sigma(H_0))< c(A) \eps^{\frac 13}
\]
for a constant $c(A)$ depending only on $A$. In particular, we have
\[
d_{GH}(\sigma(H_0), \sigma(H_1))= o(1),
\]
where $o(1) \to 0$
as $\eps\to 0$.
\end{thm}

\begin{proof}
We start by taking a point $0\leq \lambda\leq A$ in the spectrum of $H_1$. By Corollary~\ref{cor21} for $m=1,2$ we have
\begin{equation*}
|\,((H_1+1)^{-m}\psi_j,(H_1-\lambda) \psi_j)_1 \,|\leq\eps\|\psi_j\|_1
\end{equation*}
for a sequence $\{\psi_j\}$ with unit norm as $j \to \infty$.  The identity
\begin{equation} \label{gap0}
\begin{split}
&((H_1+1)^{-m}\psi_j,(H_1-\lambda) \psi_j)_1 \\
&= ( (H_1+1)^{-m+1}\psi_j,\psi_j)_1  -(1+\lambda)\, (\,(H_1+1)^{-m}\psi_j, \psi_j)_1
\end{split}
\end{equation}
together with Lemma~\ref{resBd} imply that the corresponding expression for $H_0$ should  also tend to zero. In fact, we have
\[
|((H_0+1)^{-m}\psi_j,(H_0-\lambda) \psi_j)_0 |\leq c(A)\eps
\]
for some constant $c(A)$ depending only on $A$. By Corollary ~\ref{cor21}, the conclusion is true for some, possibly different, constant $c(A)$.\end{proof}

Theorem \ref{PropBd} demonstrates that whenever $H_\eps\to H_0$ under the topology of $\eps$-closeness, then $\sigma(H_\eps)\to \sigma(H_0)$ as pointed metric spaces  in the Gromov-Hausdorff distance.  At the same time, it implies that gaps in the spectrum of $H_0$, if they exist, do not vanish instantaneously.

\section{The Spectrum of the Laplacian on $k$-Forms} \label{SecForm}

Let $(M^n,g)$ be a complete $n$-dimensional Riemannian manifold. The metric $g$ induces a pointwise inner-product on the space of $k$-forms $\Lambda^k(M)$ which is denoted $\< \cdot ,\cdot \>$. We denote the $L^2$ pairing  as $$ (\cdot , \cdot )=\int_M \< \cdot  ,\cdot \>$$  and the $L^2$ norm as $\|\cdot\|$.

Let $L^2(\Lambda^k(M))$  denote the space of $L^2$ integrable $k$-forms.  Denote by $\Delta_k$  the Laplacian on $k$-forms  as well as its Friedrichs extension on $L^2$. We denote the  domain of the Laplacian on $k$-forms by $\mathfrak{Dom}(k,\Delta)$.   For the remaining of this paper, we shall write $\Delta$ instead of $\Delta_k$ for $0\leq k\leq  n$ whenever the order of the form is implied.

As mentioned in the Introduction, the notation $\sigma(k,\Delta)$, $\sigma_{\mathrm ess}  (k,\Delta)$ \emph{resp.}, actually refers to the pointed complete metric space
\[
(\sigma(k,\Delta)\cup \{-1\}, -1),\qquad (\sigma_{\mathrm ess}  (k,\Delta)\cup\{-1\},-1)\quad {resp.}
\]
We set
\[
\sigma  (-1,\Delta)=\sigma  (n+1,\Delta)=\sigma_{\mathrm ess}   (-1,\Delta)=\sigma_{\mathrm ess}   (n+1,\Delta)=\emptyset
\]
which, according to the above convention, means that they are all the single point metric space $\{-1\}$.

\begin{thm} \label{thm31} Let $(M,g)$ be a complete Riemannian manifold. For any $0\leq k\leq n$,
suppose that $\lambda>0$ belongs to $\sigma (k,\Delta)$. Then one of the following holds:
\begin{enumerate}
\item[(a)] $\lambda\in\sigma  (k-1,\Delta)$, or
\item[(b)] $\lambda\in \sigma  (k+1,\Delta)$.
\end{enumerate}
The same result is true for the essential spectrum.
\end{thm}
\begin{proof} Let $\lambda>0$ and $\lambda\in \sigma (k,\Delta)$. By the classical Weyl criterion we know that for each $\eps>0$,  there exists an approximate eigenfunction $\omega_\eps \in \mathfrak{Dom} (k,\Delta)$ such that $\|\omega_\eps\| =1$,
\begin{equation} \label{e1b}
\|(\Delta -\lambda)\omega_\eps\| \leq\eps.
\end{equation}
As $M$ is complete, we can in fact assume that the $\omega_\eps$ are smooth and compactly supported. Choosing $\eps<\lambda/2$, the triangle inequality gives
\begin{equation} \label{e0b}
(\Delta\omega_\eps,\omega_\eps) \geq  \frac 12 \, \lambda.
\end{equation}
Thus we  have
\begin{equation*}
\|d \omega_\eps\|^2 + \|\delta \omega_\eps\|^2 = (\Delta \omega_\eps, \omega_\eps) \geq \frac 12 \,  \lambda.
\end{equation*}
This estimate implies that either
\begin{equation*}
\| d\omega_\eps\|^2  \geq \frac{\lambda}{4}  \ \ \ \  \text{or}    \ \ \ \  \| \delta \omega_\eps\|^2  \geq  \frac{\lambda}{4} .
\end{equation*}

We first consider the case  $\| d\omega_\eps\|^2  \geq \frac{\lambda}{4} $. For simplicity, we denote $\omega_\eps=\omega.$  For any integer $m$,  $ \ \ (\Delta +1)^{-m} d  \omega = d (\Delta +1)^{-m}\omega \ $  and $ (\Delta +1)^{-m} \delta \omega= \delta (\Delta +1)^{-m}  \omega \ $.
For $m=1,2$  we compute
\begin{equation}\label{e3b}
\begin{split}
|\,((\Delta +1)^{-m}\,  d \omega, \;(\Delta -\lambda)\,  d\omega)\,| & = |\,((\Delta +1)^{-m}\,  \delta d \omega,  (\Delta - \lambda) \omega)   \,|\\
& \leq \eps  \|\,((\Delta +1)^{-m}\, \delta d \omega \|
\end{split}
\end{equation}
by   \eref{e1b}. At the same time, the commutativity properties of the resolvent and integration by parts give
\begin{equation*}
\begin{split}
\|(\Delta +1)^{-m}\, \delta d \omega \|^2 + \|(\Delta +1)^{-m}\, d \delta  \omega \|^2 = \|(\Delta +1)^{-m}\,  \Delta  \omega  \|^2\leq \|\Delta \omega\|^2 \leq (\eps+\lambda)^2,
\end{split}
\end{equation*}
where we have used  Lemma \ref{lem1b} and assumption \eref{e1b}. Combining this with \eref{e3b} we get
\begin{equation*}
|\,((\Delta +1)^{-m}\,  d \omega, \;(\Delta -\lambda)\,  d\omega)\,|  \leq \eps   (\eps +\lambda) \leq \eps (1 +\lambda) \|\omega\|^2\leq\frac{4\eps (1 +\lambda)}{\lambda}\|d\omega\|^2
\end{equation*}
by our assumption.

If we consider instead the case  $\| d\omega_\eps\|^2  \geq \frac{\lambda}{4} $, we similarly get
\begin{equation*}
|\,((\Delta +1)^{-m}\,  \delta \omega, \;(\Delta -\lambda)\,  \delta\omega)\,|   \leq \frac{4\eps (1 +\lambda)}{\lambda} \|\delta\omega\|^2.
\end{equation*}
By Corollary~\ref{cor21}, $\lambda$ must therefore belong to either $\sigma(k-1,\Delta)$ or $\sigma(k+1,\Delta)$.

The case for the essential spectrum follows similarly.
\end{proof}

\begin{remark}
Over a compact manifold the $k$-form spectrum is discrete, and each element of the spectrum is an eigenvalue. To each eigenvalue $\lambda$ corresponds a smooth form $\omega$ such that
\[
\Delta\omega-\lambda\omega=0.
\]
It is easy to check that
\[
\Delta d\omega-\lambda d\omega=0,\quad \Delta \delta\omega-\lambda \delta\omega=0.
\]
Therefore if $\lambda\neq 0$, at least one of $d\omega$ and $\delta\omega$ should not be zero, and hence the conclusion of  Theorem~\ref{thm31} is trivially true. On the other hand, it seems that in order to prove the result in the complete noncompact case, we need to make full use of our new Weyl criterion.
\end{remark}

\begin{remark}  Gromov and Shubin   proved   in~\cite{GrS} that over any Riemannian manifold, including the noncomplete case, we have the following Hodge decomposition theorem:
\[
L^2(M)={\rm ker}\,\Delta\oplus {\rm Im}\, d\oplus {\rm Im}\,\delta.
\]
However, the completeness assumption in Theorem~\ref{thm31} is essential as there exists a counterexample in the incomplete case. See Lu-Xu~\cite{luxu} for details.
\end{remark}

\begin{corl} \label{cor31}
The spectrum of the Laplacian on 1-forms contains the  spectrum of the Laplacian on functions except possibly for the point $\lambda=0$.
\end{corl}

\begin{corl}\label{cor32}
The essential spectrum of the Laplacian on 1-forms is $[0,\infty)$ whenever the essential spectrum of the Laplacian on functions is $[0,\infty)$.
\end{corl}

Corollary \ref{cor31} is an immediate consequence of Theorem \ref{thm31}. Corollary \ref{cor32} follows from Corollary \ref{cor31} and the fact that the essential spectrum is a closed set. Combined with Theorem 1.3 of our article~\cite{char-lu-1} on the function spectrum these corollaries will give that the essential spectrum on 1-forms is $[0,\infty)$ on a significantly larger class of manifolds.

We recall the following definitions from ~\cite{char-lu-1}. Let $p$ be a fixed point in $M$. The cut locus  ${\rm Cut}(p)$  is a set of measure zero in $M$, and the manifold can be written as the disjoint union $M=\Omega\cup {\rm Cut}(p)$, where $\Omega$ is star-shaped with respect to $p$. That is, if $x\in \Omega$, then the geodesic line segment $\overline {px}\subset \Omega$. $\p r= \p /\p r$ is well defined on  $\Omega$.

\begin{Def} \label{defAsy}
Let $M$ be a complete noncompact Riemannian manifold. Let $p$ be a fixed point in $M$ and define $r(x)$ to be the radial function with respect to $p$. We say that the radial Ricci curvature of $M$ is asymptotically nonnegative with respect to $p$ if there exists a continuous function $\delta(r)$ on $\mathbb{R}^+$ such that
\begin{enumerate}
\item [(i)] ${\displaystyle \lim_{r\to\infty} \delta(r)=0};$
\item [(ii)] $\delta(r)>0$, and
\item[(iii)] ${\rm Ric}(\p r, \p r)\geq -(n-1) \delta(r)$ on $\Omega$.
\end{enumerate}
\end{Def}

We remark that manifolds that satisfy the condition above have subexponential volume growth at $p$, but need not have uniformly subexponential volume growth as defined in ~\cite{sturm}. In other words, the $L^p$ independence result  for the spectrum of the Laplacian on 1-forms need not hold~\cite{CharJFA}. Using Corollary \ref{cor31} and Theorem 1.3 of \cite{char-lu-1} we obtain
\begin{thm} \label{thm1forms}
Let $M$ be a complete noncompact Riemannian manifold. Suppose that, with respect to a fixed point $p$, the radial Ricci curvature is asymptotically nonnegative in the sense of Definition \ref{defAsy}. If the volume of the manifold is finite we additionally assume that its volume  does not decay exponentially at $p$.

Then the essential spectrum of the Laplacian on 1-forms is $[0,\infty)$.
\end{thm}

By the Poincar\'e duality, it is easy to see that
\[
\sigma(k,\Delta)=\sigma(n-k,\Delta)
\]
for any $0\leq k\leq n$.   Therefore, everything we stated for $1$-forms will also be true for $n-1$ forms. On the other hand, even though the $1$-form spectrum essentially contains the function spectrum, there is no monotonicity for the $k$-form spectrum with respect to $k$. In the case of hyperbolic space we have  {\em unimodality}, which means that the spectrum is increasing for $k\leq n/2$ and decreasing for other $k$:
\begin{example}\label{remark47}  The essential spectrum of the Laplacian on forms over hyperbolic space $\mathbb{H}^{N+1}$ is given by
\[
\sigma_{\mathrm ess} (k,\Delta) = \sigma_{\mathrm ess} (N+1-k,\Delta) =[\,(\frac{N}{2} - k)^2, \infty\,)
\]
for $0\leq k \leq \frac{N}{2}$, and whenever $N$ is odd
\[
\sigma_{\mathrm ess} (\frac{N+1}{2},\Delta) = \{0\} \cup [\,\frac{1}{4}, \infty\,).
\]
A proof of this result can be found in Donnelly~\cite{Don2}. Mazzeo and Phillips show in \cite{mazz} that the same result is true over quotients of hyperbolic space, $\mathbb{H}^{N+1}/\Gamma$, that are geometrically finite and have infinite volume.
\end{example}

However, one cannot expect a unimodality result on every manifold as we can see from the following example:
\begin{example}
Consider the product manifold $M^4= F^3 \times \mathbb{R}$, where $F^3$ is the compact flat three-manifold constructed by Hantzsche and Wendt in 1935 with first Betti number zero (see \cite{Cobb} for a family of manifolds of any dimension $n\geq 3$ with the same property). Note that $M$ is a flat noncompact manifold. By Theorem \ref{thm31} and  Definition \eref{lem81}
\[
\sigma_{\mathrm ess} (k,\Delta)=[0,\infty) \qquad \text{for} \ \ k=0,1,3,4.
\]
However, since there do not exist any harmonic 1-forms nor harmonic 2-forms on $F$ then
\[
\sigma_{\mathrm ess} (2,\Delta)=[a,\infty) \qquad \text{for some} \ \ a>0.
\]
In other words, its essential spectrum is smaller in half-dimension. Note that this does not contradict the result of Theorem \ref{thm31}.
\end{example}

 The computation of the $k$-form spectrum of the Laplacian for  $k\neq 0, 1,n-1, n$  is a  significantly more difficult task compared to the function case, due to the increased complexity in obtaining and controlling  approximate eigenforms. The classical Weyl criterion was used  to compute the $k$-form essential spectrum  over hyperbolic manifolds~\cites{Don2,mazz}, and over warped product measures with negative curvature~\cite{Ant} (for the Laplacian on functions see~\cite{Don81}).  On manifolds with asymptotically nonnegative curvature, it was not possible to apply the Weyl criterion directly without much stronger assumptions on the geometry and curvature of the manifold to show that the function spectrum was $[0,\infty)$ (see for example~\cites{ElwW04,EF93,Esc86,EF92,Don}). The $L^p$ independence result of Sturm~\cite{sturm} allowed for the computation of the $L^2$ spectrum under more general geometric conditions, which was first done by~\cite{Wang97} and later improved by ~\cite{Lu-Zhou_2011}. The $L^p$ independence result for forms made the computation of the $L^2$ 1-form spectrum a bit simpler, but it still required stronger curvature assumptions than for the case of functions to get the spectrum equal to $[0,\infty)$ (cf. ~\cite{CharJFA}). Theorem \ref{thm1forms} demonstrates the strength of the analytic result Theorem \ref{thm31}.

There exist various examples of manifolds over which the essential spectrum of the Laplacian on functions has gaps.  Lott proved in \cite{Lott2} that for any $\epsilon>0$, there is a complete connected noncompact finite-volume Riemannian manifold whose sectional curvatures lie in $[-1-\epsilon,-1+\epsilon]$ and whose function Laplacian has an infinite number of gaps in its essential spectrum. The gaps tend towards infinity. In the same article Lott claims that there should even be examples where the essential spectrum is a Cantor set. Post and  Lled\'o  \cites{Post1,Post2} use Floquet theory to give examples of Riemannian coverings of manifolds whose  essential spectrum has at least a prescribed finite number of gaps. More recently,  Schoen and  Tran in \cite{SchoTr} show that for any noncompact covering of a compact manifold one can find a metric on the base so that the lifted metric has an arbitrarily large number of gaps in its essential spectrum. Moreover, for any complete noncompact manifold with bounded curvature and positive injectivity radius they find a  uniformly equivalent metric with an arbitrarily large number of gaps in the essential spectrum. Their manifolds do not have nonnegative Ricci curvature, but some of them  have bounded positive scalar curvature.

Corollary \ref{cor31} tells us that the essential spectrum on 1-forms over these manifolds could be larger, but it is not known whether it would also have gaps. On the other hand, recall that on both the hyperbolic space, for $k\neq (N+1)/2$, (constant negative curvature) and the Euclidean space (zero curvature) the essential spectrum of the Laplacian on forms is a connected subset of the real line. It would be quite interesting to find sufficient conditions on the geometry of the manifold so that its essential spectrum is a connected subset of the real line.

\section{The Spectrum of the Laplacian on Forms Under Continuous Deformation of Metrics} \label{S5}

Let $(M,g)$ be a complete noncompact Riemannian manifold.   We  consider  the two operators that make up the Laplacian $\Delta$. We set
\[
\mathcal{L}^1=\delta d
\]
with associated quadratic form $Q^1 (\omega)=(d \omega,d \omega)$ and
\[
\mathcal{L}^2=   d \delta
\]
with associated quadratic form $Q^2 (\omega)=(\delta \omega,\delta \omega)$ on $k$-forms.

Each one of the above operators has a self-adjoint Friedrichs extension which is nonnegative. It can be easily seen that
$\mathfrak{Dom}(k,\Delta)= \mathfrak{Dom}(k,\mathcal{L}^1)\cap  \mathfrak{Dom}(k,\mathcal{L}^2)$. We will now illustrate how the spectra of these three operators are related, which is of its own interest.
\begin{lem} \label{lemSp} For any $0\leq k \leq n$ the following containments hold
\begin{equation} \label{Sp1}
\sigma(k,\Delta) \subset \sigma  (k,\mathcal{L}^1)\cup \sigma  (k,\mathcal{L}^2)
\end{equation}
and
\begin{equation}\label{Sp2}
\sigma (k,\Delta) \supset \{\sigma   (k,\mathcal{L}^1)\cup \sigma  (k,\mathcal{L}^2)\} \setminus \{0\}.
\end{equation}
The result is also true for the essential spectra of the operators.
\end{lem}
\begin{proof}

We first remark that  $\Delta, \mathcal{L}^1,\mathcal {L}^2$ are all closed. Therefore  the forms to which we apply the Weyl Criteria can be taken to be  smooth with compact support.

If $k=0$, then $\Delta=\mathcal L^1$ and $\mathcal L^2=0$, and the statement is trivially true. Similarly, for $k= n$, $\Delta=\mathcal L^2$. As a result we only consider the case $0<k<n$.

We begin by proving \eref{Sp1}.  We first show that 0 is always a point in $ \sigma  (k,\mathcal{L}^1)\cup \sigma  (k,\mathcal{L}^2).$ This follows from the simple fact that for any smooth compactly supported $(k-1)$-form $\omega$, $d\omega$ is a $k$-eigenform of $\mathcal{L}^1=\delta d$ corresponding to the eigenvalue 0.  Moreover, since $k\geq 1$ we can always find a sequence of compactly supported approximate $(k-1)$-forms $u_j$ such that $\|du_j\|=1$ on $M$. This implies that $0\in \sigma   (k,\mathcal{L}^1)$. As a result, if $0 \in \sigma(k,\Delta)$, then $0\in \sigma  (k,\mathcal{L}^1)\cup \sigma  (k,\mathcal{L}^2) $.

We now consider  $\lambda>0$ in $\sigma(k,\Delta)$. By the classical Weyl criterion, there exists  a sequence of approximate eigenforms $\{\psi_j\}_{j\in \mathbb{N}}$ with $\| \psi_j\|=1$ such that for any $\eps>0$, $0<\eps<\lambda/2$, we have
\[
\|(\Delta -\lambda) \psi_j\| <\eps \ \ \ \text{as} \ \ \ j\to \infty.
\]
By the triangle inequality,
\[
\|\Delta \psi_j\| \geq \frac{\lambda}{2}
\]
for $j$ large enough. Since\footnote{Note that the $\psi_j$ are smooth with compact support.} $\|\Delta \psi_j\|^2 = \|\mathcal{L}^1  \psi_j\|^2 +\|\mathcal{L}^2 \psi_j\|^2 $, there must exist a subsequence of $j$, for which either
\[
\|\mathcal{L}^1  \psi_j\|^2 \geq \frac{\lambda}{4} \ \ \text{or} \ \ \|\mathcal{L}^2  \psi_j\|^2 \geq \frac{\lambda}{4}.
\]

Suppose that $\|\mathcal{L}^1  \psi_j\|^2 \geq \frac{\lambda}{4}$. Observe that  on smooth forms with compact support,
\[
\mathcal{L}^1  \mathcal{L}^2 = \mathcal{L}^2 \mathcal{L}^1 =0
\]
and
\[
\Delta \mathcal{L}^i = \mathcal{L}^i \Delta
\]
for $i=1,2$.  Thus for $m=1,2$
\begin{equation*}
\begin{split}
&|\,(\,(\mathcal{L}^1+1)^{-m} \mathcal{L}^1 \psi_j,  (\mathcal{L}^1-\lambda) \mathcal{L}^1 \psi_j) \,|
= |\ (\,(\mathcal{L}^1+1)^{-m} \mathcal{L}^1 \psi_j,(\Delta  -\lambda) \mathcal{L}^1 \psi_j) \,| \\&=  |\ (\,(\mathcal{L}^1+1)^{-m} (\mathcal{L}^1)^2 \psi_j,(\Delta  -\lambda) \psi_j) \,|
\leq \| \mathcal{L}^1 \psi_j \|\; \cdot\|(\Delta  -\lambda) \psi_j\| \leq   \eps \, \frac{4}{\lambda}\,\| \mathcal{L}^1 \psi_j \|^2
\end{split}
\end{equation*}
where we have used  that $\|(\mathcal{L}^1+1)^{-m}\mathcal{L}^1\|\leq 1$  which can be proved similarly to Lemma \ref{lem1b}. Setting $\tilde\psi_j=\mathcal{L}^1 \psi_j/\|\mathcal{L}^1 \psi_j \|$ and rescaling the above inequalities, we see that the $\tilde\psi_j$   satisfy the conditions of Corollary \ref{cor21}. Therefore $\lambda\in \sigma (k,\mathcal{L}^1)$.  The argument for the case $\|\mathcal{L}^2  \psi_j\|^2 \geq \frac{\lambda}{4}$ is identical. We thus conclude that  $\lambda$ belongs either to $\sigma (k,\mathcal{L}^1)$ or to  $\sigma (k,\mathcal{L}^2)$.

To prove \eref{Sp2} we now suppose that $\lambda>0$ belongs to $\sigma (k,\mathcal{L}^1)$.  Again by the classical Weyl criterion there exists a sequence of smooth approximate eigenforms $\{\psi_j\}_{j\in \mathbb{N}}$ with $\| \psi_j\|=1$ such that for any $\eps>0$, $0<\eps<\lambda/2$, we have
\[
\|(\mathcal{L}^1 -\lambda) \psi_j\| <\eps \ \ \ \text{as} \ \ \ j\to \infty.
\]
As a result,
\[
\frac{\lambda}{2}\leq \|\mathcal{L}^1 \psi_j \|\leq 2\lambda
\]
for $j$ large enough. For $m=1,2$ we similarly get
\begin{equation*}
\begin{split}
&|\,(\,(\Delta +1)^{-m} \mathcal{L}^1 \psi_j,  (\Delta-\lambda) \mathcal{L}^1 \psi_j)\,|
= |\, (\,(\Delta+1)^{-m} \mathcal{L}^1 \psi_j,(\mathcal{L}^1  -\lambda) \mathcal{L}^1 \psi_j)\,|\\
& = |\, (\,(\Delta+1)^{-m} \Delta \mathcal{L}^1 \psi_j,(\mathcal{L}^1  -\lambda) \psi_j)\,|
\leq \| \mathcal{L}^1 \psi_j \|\;\cdot \|(\mathcal{L}^1  -\lambda) \psi_j\| \leq \eps \,\frac{4}{\lambda} \,  \| \mathcal{L}^1 \psi_j \|^2.
\end{split}
\end{equation*}
Setting  $\tilde\psi_j=\mathcal{L}^1 \psi_j/\|\mathcal{L}^1 \psi_j \|$ and rescaling the above inequalities,   we see that the $\tilde\psi_j$  satisfy the conditions of Corollary \ref{cor21}. Therefore, $\lambda$ belongs to the spectrum of $\Delta$.

In a similar manner we can prove that $\sigma (k,\mathcal{L}^2) \setminus \{0\} \subset \sigma (k,\Delta)$. As a result
\[
\{\sigma (k,\mathcal{L}^1) \cup \sigma (k,\mathcal{L}^2)\} \setminus \{0\} \subset  \sigma (k,\Delta).
\]
The case of the essential spectrum follows in a similar manner.
\end{proof}

We now consider a manifold $M$ to which we can assign two Riemannian metrics, $g_0, g_1$ such that $(M,g_0)$ and $(M,g_1)$ are smooth complete manifolds with respect to both. We say that the two metrics are $\eps$-close if for some $0<\eps<1$
\begin{equation} \label{ecls}
(1- \eps) g_0  \leq g_1 \leq  (1+\eps) g_0.
\end{equation}

We denote by $\delta_i$ the adjoint of $d$ on $(M,g_i)$ for $i=0,1$ and the associated Laplacian operators by
\[
\Delta_i =d \delta_i+ \delta_i d.
\]

We denote by $(\cdot,\cdot)_i$ the $L^2$ pairing in the $g_i$ metric and by $\|\cdot\|_i$ the respective $L^2$ norm. In this section we will show that Theorem \ref{PropBd} can be extended  to $\Delta_i$. Let
\[
\mathcal{L}^1_i=\delta_i d \ \ {\rm  and} \ \ \mathcal{L}^2_i=   d \delta_i.
\]
Their associated quadratic forms are given by
\[
Q^1_i (\omega,\omega)=(d \omega,d \omega)_i \ \ {\rm and} \ \ Q^2_i (\omega,\omega)=(\delta_i \omega,\delta_i \omega)_i,
\]
respectively.

\begin{thm} \label{thmCont}
Let $M$ be a  manifold, and let $g_0, g_1$ be two smooth complete Riemannian metrics on $M$ that are $\eps$-close for some $0<\eps<1/2$.

Fix $A>0$. Then for any $\lambda \in \sigma(k,\Delta_1)\cap [0,A]$
\[
{\rm dist}(\lambda,\sigma(k,\Delta_0)) <   c(A) \, \eps^{\frac 13}
\]
for some constant $c(A)$ depending only on $A$.
A similar result holds for the essential spectra of the operators.
In particular,
\[
d_{GH}(\sigma(k,\Delta_1), \sigma(k,\Delta_0))=o(1),
\]
where $o(1)\to 0$, as $\eps\to 0$.
\end{thm}
\begin{proof}
 Given that the $*$-operator is an isometry in the respective metric,  for all $0\leq k \leq n$, we have
\[
\mathfrak{Dom}(k,\mathcal{L}^2_i) = \mathfrak{Dom}({ n -k},\mathcal{L}^1_i)
\]
for $i=0,1$. Moreover,
\[
\sigma(k,\mathcal{L}^2_i) =\sigma(n -k,\mathcal{L}^1_i).
\]
The same holds true for the essential spectrum.

Since $d$ is metric independent, and  since $\mathfrak{Dom}(k,\mathcal{L}^1_1)\cap \mathfrak{Dom}(k,\mathcal{L}^1_0)\supset\mathcal C$, where $\mathcal C$ is the space of smooth forms with compact support, by the $\eps$-closeness of the metrics we have
\[
(1-\eps) \, Q^1_0 (\omega) \leq Q^1_1 (\omega)\leq (1+\eps)\, Q^1_0 (\omega)
\]
for any $\omega\in\mathcal C$. In other words, the operators $\mathcal{L}^1_1$ and $\mathcal{L}^1_0$ are $\eps$-close. For  any $\lambda \in \sigma(k,\mathcal{L}^1_1)\cap (0,A]$, by  Lemma \ref{lemSp}, we have
\[
{\rm dist}(\lambda,\sigma(k,\Delta_0))\leq {\rm dist}(\lambda,\sigma(k,\mathcal{L}^1_0)) <   c(A) \, \eps^{\frac 13}.
\]

The case $\lambda=0$ can be treated directly. If $0\in\sigma(\Delta_1)$, we claim
\[
{\rm dist}(0,\sigma(\Delta_0))\leq 100\eps.
\]
If not, then there exists $\eps_0$ such that ${\rm dist}(0,\sigma(\Delta_0))>100\eps_0$. Let $\omega$ be a smooth compactly supported $k$-form such that $\|\omega\|_1=1$ and
\[
\|\Delta_1\omega\|_1<\eps_0.
\]
Then
\[
\|d\w\|_1 \leq \sqrt{\eps_0}, \ \ \|\delta_1 \w\|_1 \leq \sqrt{\eps_0}.
\]

Set $\eta=\Delta_0^{-1}\omega, \ \ \eta_1=d\eta$ and $\eta_2=\delta_0 \eta$. Then
\[
\|\eta\|_0 \leq \frac{1}{100 \eps_0}, \ \ \|\eta_1\|_0 \leq \frac{1}{5 \sqrt{\eps_0}}, \ \ \|\eta_2\|_0 \leq  \frac{1}{5 \sqrt{\eps_0}}, \ \ \text{and} \ \ \|d \eta_2\|_0 \leq 2.
\]
Therefore,
\[
|(\w, \delta_0 \eta_1)_0 | = |(d\w, \eta_1)_0| \leq \|d\w\|_0 \cdot \|\eta_1\|_0 \leq 2 \|d\w\|_1 \cdot \|\eta_1\|_0 \leq \frac{2}{5}.
\]
Moreover,
\[
|(\w, d \eta_2)_1 | = |(\delta_1\w, \eta_2)_1| \leq  \|\delta_1\w\|_1 \cdot 2 \|\eta_2\|_0 \leq \frac{2}{5},
\]
whereas by the $\eps$-closeness of the metrics we have
\[
|(\w, d \eta_2)_1 -(\w, d \eta_2)_0 | \leq 2 \eps \|d \eta_2\|_0 \leq 4 \eps.
\]
As a result,
\[
|(\w, d \eta_2)_0 | \leq \frac{2}{5} +4\eps.
\]

Observing that $\w= \delta_0 \eta_1 + d \eta_2$ and combining the above estimates we get
\[
\|\w\|_1= |(\w, \delta_0 \eta_1)_0 + (\w, d \eta_2)_0  | \leq \frac 45 + 4\eps < \frac{9}{10}
\]
which gives a contradiction for $\eps$ small enough. This completes the proof.
\end{proof}

The following corollary is now immediate.
\begin{corl}
Let $M$ be a complete noncompact manifold, and let $\{g_\eps\}_{\eps\in[0,1/2]}$ be a family of smooth complete Riemannian metrics on $M$ such that
\[
(1- \eps) g_0  \leq g_\eps \leq  (1+\eps) g_0.
\]
Then
\[
d_{GH}(\sigma(k,\Delta_\eps), \sigma(k,\Delta_0))=o(1),
\]
where $o(1)\to 0$, as $\eps\to 0$. A similar result holds for the essential spectrum of the operators.
\end{corl}

The Corollary implies that if for a sequence of $\eps_m\to 0$ there exists a sequence of points $\lambda_{\eps_m} \to \lambda>0$ with the property $\lambda_{\eps_m}\in \sigma_\mathrm{ess}(\Delta_{\eps_m})$ for all  $m$,  then  $\lambda\in \sigma_\mathrm{ess}(\Delta_0)$. In other words, if $\lambda\notin \sigma_\mathrm{ess}(\Delta_0)$, then there exists a $\delta>0$ such that $\lambda\notin \sigma_\mathrm{ess}(\Delta_\eps)$ for all $\eps<\delta$.

When $M$ is a compact manifold the Gromov-Hausdorff convergence of the spectrum implies that for any given $k$, the $k$-th eigenvalues are convergent as $\eps \to 0$. Therefore, we recover Dodziuk's result~\cite{dod}.

\section{A Localization Theorem}\label{s7}

In the remaining sections we be will considering the $k$-form spectrum of manifolds which are asymptotically flat or whose Ricci curvature is asymptotically nonnegative.  Here we will show that the bottom of the essential spectrum of such manifolds is reflected in a sequence of large geodesic balls.

We first introduce the Gromov cover. Let $\Omega\subset \tilde \Omega$ be two open sets of a complete Riemannian manifold $M$. Assume that on $\tilde \Omega$, the Ricci curvature satisfies $\mathrm{Ric}\geq -(n-1) R^{-2}$ for some $R>1$. Assume that
\[
{\rm dist}(\pa\tilde \Omega, \Omega)>0.
\]
Let $0<R< {\rm dist}(\pa\tilde \Omega, \Omega)$.
Let $x_1,\cdots, x_N$ ($N$ may be infinite) be a  set of points in $\Omega$ such that
\[
d(x_i,x_j)\geq R \ \ \text{for all} \ \ i\neq j, \ \  \text{and}  \ \ \bigcup_{i=1}^N B_{x_i}(R) \supset \Omega.
\]
Such a cover, which always exists, is called the Gromov cover.
Let $y$ be a point in the intersection of $s$ of the balls $B_{x_i}(2R)$,
\[
y\in B_{x_{i_1}}(2R) \cap \cdots \cap B_{x_{i_s}}(2R).
\]
Then, since the $B_{x_i}(R/2)$ are disjoint, we have
\[
\sum_{j=1}^s \text{Vol}( B_{x_{i_j}}( R/2) ) \leq \text{Vol}( B_y( 5R/2) ).
\]
At the same time, $B_y( 5R/2)\subset B_{x_{i_k}}( 9R/2)  $ for any $k=1,\cdots, s$ and by the Bishop-Gromov volume comparison theorem (and our Ricci curvature assumption), we get
\[
\text{Vol}( B_{x_{i_k}}( 9R/2) ) \leq C  \; \text{Vol}( B_{x_{i_k}}( R/2) )
\]
for some uniform constant $C=C(n)$, and for any $k=1,\cdots, s$. Combining the above we obtain
\[
\sum_{j=1}^s \text{Vol}( B_{x_{i_j}}( R/2) ) \leq  C  \; \text{Vol}( B_{x_{i_k}}( R/2) )
\]
and therefore $s\leq C=C(n)$.

For each $i=1,\cdots, N$ we define the cut-off functions $0\leq\varphi_i(x)\leq 1$ such that
\begin{equation*}
\begin{split}
& \varphi_i(x)= 1 \ \ \text{whenever  } d(x,x_i)\leq R, \\
& \varphi_i(x)= 0 \ \ \text{whenever  } d(x,x_i)\geq 2R, \ \ \text{and} \\
& |\n \varphi_i| \leq C/R
\end{split}
\end{equation*}
for some uniform constant $C$.

We let
\begin{equation} \label{puni}
\rho_i (x) = \frac{\varphi_i(x)}{\sqrt{\sum_{j=1}^N\varphi^2_j(x)}}
\end{equation}
and observe that $\sum_{i=1}^N \rho_i^2(x) =1$ on $M$ and $ \sum_{i=1}^N \varphi_i^2(x) \geq 1$ for $x\in \Omega$. Since the supports of the $\varphi_i(x)$ cover each point of $\Omega$ at most $C(n)$ times, we can also show that for some uniform constant $C$
\[
|\n \rho_i (x)|\leq   C /R
\]
for all $x\in \Omega$.

Setting $\w_i = \rho_i \w$, we have
\begin{equation*}
\begin{split}
\|\n \w_i\|^2 =  (|\n \rho_i|^2 \w,\w) + 2(\nabla\rho_i \omega, \rho_i\nabla\,\omega)+ \; (\rho_i^2 \,\n \w, \n \w).
\end{split}
\end{equation*}
Since the supports of the $\rho_i$ also cover each point of $\Omega$ at most $C(n)$ time, we have
\begin{equation}\label{2-est}
\sum_{i=1}^N \|\n \w_i\|^2  \leq    \frac {C}{R^2} \, \| \w \|^2 +    \|\n \w\|^2
\end{equation}
for  a uniform constant $C$, where we have used the fact that $\sum_{i=1}^N \rho_i^2=1$ and, as a result, $\; 2 \sum_{i=1}^N \rho_i \n \rho_i = 0$.

Denote by $\lambda_o^\mathrm{ess}(k,\Delta,M)$ the bottom of the essential spectrum of the Laplacian on $k$-forms over $M$.

For any open  subset $\Omega\subset M$ we consider the Friedrichs extension of the Laplacian on $k$-forms over $\Omega$ which we denote by $\Delta$.
Let
\[
\lambda_o(k,\Delta,\Omega)=\inf_{\omega \in \mathcal C_o^{\infty}(\Lambda^k(\Omega))} \frac{(\Delta \omega, \omega)}{\|\omega\|^2}
\]
be the infimum of the Rayleigh quotient for $\Delta$ over compactly supported forms in $\Omega$.

\begin{lem} \label{lemSp2}
Let $M$ be a complete noncompact Riemannian manifold whose Ricci curvature is asymptotically nonnegative at infinity. Let $R>1$  be a positive constant and  $0\leq k\leq n$. Consider the open  subset of $M$
\begin{equation}\label{volume-2}
\Omega_R=\{x\in M\, \big| \, \  \mathrm{Ric} (x) \geq -(n-1) R^{-2}  \,\}.
\end{equation}
Then there exists a point $x\in \Omega_R$  and a uniform constant $C(n)$ such that
\[
\lambda_o (k,\Delta,B_{x}(R)) \leq  \lambda_o^\mathrm{ess}(k,\Delta,M)   +   \frac {C(n)}{R^2}.
\]
We can in fact choose $x$ such that $B_{x}(R)\subset \Omega_R$.
\end{lem}

\begin{proof}
Since the Ricci curvature is asymptotically nonnegative,  $\Omega_R$ is  the complement of a compact subset of $M$. Therefore
\[
\lambda_o^\mathrm{ess}(k,\Delta, M)=\lambda_o^\mathrm{ess}(k,\Delta,\Omega_R).
\]

Let $\Omega$ be a compact domain of $\Omega_R$ whose diameter is larger than $R$ and such that $\{ x\in M \ | \ d(x,\Omega) \leq 5R\} \subset \Omega_R$. We take a Gromov cover
\[
\bigcup_{i=1}^N B(x_i,R)\supset \Omega
\]
with respect to $R$. Let $\w$ be any smooth $k$-form with compact support in $\Omega$. Setting $\w_i =\rho_i \w$ as above, we get
\[
\sum_{i=1}^N (\Delta \w_i, \w_i) =  \sum_{i=1}^N \; \|\n \w_i\|^2 + (\mathcal{W}_k \w_i, \w_i)
\]
by the Weitzenb\"ock formula, where $\mathcal{W}_k$ is the Weitzenb\"ock tensor on $k$-forms. Then by ~\eqref{2-est}, we have
\begin{equation}\label{gromov}
\sum_{i=1}^N  (\Delta \w_i, \w_i)  \leq \frac{C}{R^2} \, \| \w \|^2 +    (\Delta \w, \w).
\end{equation}
Therefore,
\[
\inf_i\lambda_o(k,\Delta, B(x_i, R))\,\|\omega\|^2\leq \frac {C}{R^2} \, \|\omega\|^2+\,  (\Delta \w, \w),
\]
and the lemma is proved.
\end{proof}

We will use the above lemma to prove a stronger result for the bottom of the essential spectrum of $M$.
\begin{lem} \label{lem7b}
Let $M$ be a complete Riemannian manifold whose Ricci curvature is asymptotically nonnegative at infinity. Then there exists a sequence of pairs $S=\{(x_i,R_i)\}$ with $x_i \in M$ and $R_i \in \mathbb{R}^+$ satisfying
\begin{enumerate}
\item
$x_i, R_i   \to \infty  \ \ \text{as}  \ \ i\to \infty$,
\item
$\text{the geodesic balls} \ \  B_{x_i}(R_i) \ \ \text{are disjoint, and}$
\item
$\mathrm{Ric} \geq  - (n-1)R_i^{-2} \ \ \text{on} \,\,B_{x_i}(R_i)$.
\end{enumerate}
Then for any $R_i'<R_i$, $R_i'\to\infty$, we have
\[
\lambda_o^\mathrm{ess}(k,\Delta,M) =  \lim_{i\to \infty} \lambda_o (k,\Delta,B_{x_i}(R_i')).
\]
\end{lem}
\begin{proof}
From the definition, it is clear that
\[
\lambda_o^\mathrm{ess}(k,\Delta,M) \leq  \lim_{i\to \infty} \lambda_o (k,\Delta,B_{x_i}(R_i')).
\]
The reverse inequality follows from Lemma \ref{lemSp2}.
\end{proof}

\begin{Def} The sequence $S$ is called a minimal sequence. \label{defMin}
\end{Def}

\begin{remark}
A minimal sequence plays a very important role in locating the essential spectrum of the manifold $M$ and it in fact contains all spectral information for manifolds with computable spectrum. By its definition, the minimal sequence is not unique. In particular, by Lemma~\ref{lem7b}, there exists a lot of flexibility in choosing the  sequence $R_i'<R_i$ which can be made to go to infinity very slowly.
\end{remark}

\begin{remark}  In spite of what we had previously anticipated, the analogous result on manifolds with Ricci curvature bounded below seems not to be true, even in the case of the function spectrum. Let $M_i$ be a sequence of compact hyperbolic spaces whose diameters go to infinity. We construct a complete manifold by joining consecutive hyperbolic  manifolds in this sequence using hyperbolic necks (cf. ~\cite{SchoTr}). It is clear that the first Dirichlet eigenvalue of a ball of radius $R$ satisfies
\[
\lambda_o(B_{x_i}(R))\geq \frac 14 (n-1)^2
\]
for any fixed large number $R>0$. On the other hand, the infimum of the essential spectrum of the manifold is zero, by considering the constant function on the large compact hyperbolic manifold with two small necks attached.
\end{remark}

\section{The Spectrum of Asymptotically Flat Manifolds}\label{7}

In this section,  we consider the Gromov-Hausdorff limit of a minimal sequence on an asymptotically flat manifold $M$ and prove that this sequence captures the essential spectrum of the manifold.   In general, the limit space of these geodesic balls need only be a metric space and not necessarily smooth. However, the asymptotic flatness of $M$ gives us the flexibility to find a cone at infinity with mild singularities. We can then use the spectrum of the limit space to compute the spectrum of $M$.

Let $S=\{(x_i,R_i)\}$ be a minimal sequence for $M$ as in Lemma \ref{lem7b}. By asymptotical flatness we may assume that the curvature of these balls decays to zero at such a rate so that  the exponential map at $x_i$ is nondegenerate within the balls {$B_{x_i}(2R_i)$} (see Proposition 4.13 in \cite{CE}).  Let $I(x)$ denote the injectivity radius of a point $x$ in $B_{x_i}(R_i)$. Given the curvature bound on $B_{x_i}(2R_i)$ we apply  \cite{CLY}*{Corollary 1} with $d\leq 2R_i$ and $T=d/2$ to obtain
\begin{lem} \label{Linj}
Let $S=\{(x_i,R_i)\}$ be a minimal sequence of geodesic balls in $M$. Then there exists a universal constant $C$, such that
\[
I(x) \geq C \, I(x_i)
\]
for any point $x\in B_{x_i}(R_i)$.
\end{lem}

There are two distinct cases of convergence: the noncollapsing case (when the injectivity radius at $x_i$ is uniformly bounded below), and the  collapsing case (when the injectivity radius at $x_i$ goes to zero). We will study these two cases separately, and compute the $k$-form essential spectrum in each one. In the noncollapsing case there is a distinction among two subcases, depending on how large the injectivity radius becomes in the limit. This gives us more precise information about the spectrum.  In this section we will prove the noncollapsing cases of Theorem \ref{mainthm}. First we recall the following result from \cite{ChLu5}. For a compact manifold $S^{n-s}$ denote by $\lambda_S(l)$
the smallest eigenvalue of the Laplacian $\Delta_S$ on $l$-forms, for $0\leq l \leq n-s$. Note that  $\lambda_B(l)= \lambda_B(n-s-l)$.
\begin{Def} \label{lem81} Define $\alpha(B,s,n,k)=0$ when $s\geq n/2$, or when $s<n/2$  and  either $0\leq k\leq s$, or $n-s\leq k\leq n$. Define $\alpha(B,s,n,k)=\min\{\lambda_B(k-l)\mid 0\leq l\leq s\}$, when
$s+1\leq k\leq n/2$, and $\alpha(B,s,n,k)=\alpha(B,s,n,n-k)$, when $n/2< k\leq n-s-1$.
\end{Def}
As we proved in \cite{ChLu5}, the product manifold  $M^n=S^{n-s}\times \mathbb{R}^s$  has empty point spectrum and
\[
\sigma(k,\Delta, M)=\sigma_{\rm ess}(k,\Delta, M)=[\alpha(S,s,n,k),\infty).
\]
Moreover, we showed the following.
\begin{thm}[\cite{ChLu5}] \label{thmF2}
Let $X=\mathbb{R}^n/\Gamma$ be a flat noncompact Riemannian manifold. Then, there exists a  compact flat manifold  $S$  of dimension $n-s$  such that
\[
\sigma(k,\Delta,{\R^n}/{\Gamma}) =\sigma_{\rm ess}(k,\Delta,{\R^n}/{\Gamma}) = \sigma_{\rm ess} (k, \Delta, S^{n-s}\times  \R^{s})  = [
\alpha(S,s,n,k), \infty).
\]
\end{thm}

\begin{proof}[Proof of Theorem \ref{mainthm},   Cases (i), (ii):]

By \cite{Fuk1}*{Theorem 6.7} (see also  \cite{Gr1}*{(8.20)}),  whenever we have a convergent sequence of pointed manifolds $(M,x_i)$ with injectivity radius uniformly bounded below and bounded curvature, the limit space $(X,p)$ is a smooth manifold with a $\mathcal C^{1,\alpha}$-metric tensor.
In  Case {\it (i)} the limit space must be the pointed Euclidean space $(\mathbb{R}^n,0)$ endorsed with the  flat metric $g_E$. As a result, one can pull-back test $k$-forms from the Euclidean space into the manifold as in Theorem \ref{thmSpecGH} below, and obtain $\sigma_\mathrm{ess}(k,\Delta, M)=[0,\infty)$.

In Case {\it (ii)},  the limit space is a flat complete noncompact manifold  $(X,p)$.  Similar to Case {\it (i)},  by Theorem  \ref{thmSpecGH} below we get $[\alpha(S,s,n,k), \infty ) \subset  \sigma_\mathrm{ess}(k,\Delta, M)$ for $\alpha(S,s,n,k)$ corresponding to the flat manifold $X$ as given by Theorem~\ref{thmF2}.

Since the sequence $\{(x_i,R_i)\}$ is minimal and the curvature tensor vanishes asymptotically on the geodesic balls $B_{x_{i}}(R_{i})$, by Lemma~\ref{lemSp2}, there exist $k$-forms $\omega_i \in  \mathcal C_o^{\infty} (B_{x_{i}}(R_{i}))$  such that
\[
\lambda_o^{\mathrm{ess}}(k,\Delta,M)=\liminf_i \frac{ \| \n \omega_i\|^2}{\|\omega_i\|^2}.
\]
On the other hand, the convergence of the  $B_{x_{i}}(R_{i})$ to $X$ is in the $\mathcal C^{1,\alpha}$-sense. Therefore the bottom of the Rayleigh quotient of $B_{x_{i}}(R_{i})$ converges to bottom the Rayleigh quotient of $X$ which is $\alpha(S,s,n,k)$. Therefore, $\lambda_o^{\mathrm{ess}}(k,\Delta,M) = \alpha(S,s,n,k)$. This completes the proof of Case {\it (ii)}.
\end{proof}

\begin{thm} \label{thmSpecGH}
Let $(M,g_M)$ and $(X,g_X)$ be two complete noncompact manifolds. Assume that there exists a sequence $\{(x_i,R_i)\}$ with the property that $x_i\to \infty$ and $R_i\to\infty$,  and all the balls $B_{x_i}(R_i)$ are disjoint. Moreover, assume that for a fixed  point  $p\in X$ and a sequence of $\eps_i \to 0$ there exist differentiable maps
\[
f_i : B_{x_i}(R_i) \to B_p(R_i)
\]
such that
\begin{equation} \label{thm7e1}
|(f_i)_*(g_M)-g_X|_{g_X} < \eps_i.
\end{equation}

Then
\[
\sigma(k,\Delta_X) \subset  \sigma_\mathrm{ess}(k,\Delta_M).
\]
\end{thm}

\begin{proof}
As in the proof of Theorem \ref{thmCont}, it will suffice to show that
\[
\sigma(k,\mathcal{L}^1_X) \subset  \sigma_\mathrm{ess}(k,\mathcal{L}^1_M)
\]
where $\mathcal{L}^1_M= \delta_M d_M$ and  $\mathcal{L}^1_X= \delta_X d_X$. We will treat the case $0\in \sigma(k,\Delta_X)$ separately.

Denote the $L^2$ pairing on $(M,g_M)$ by $( \cdot \, , \cdot)_M$  and on $(X,g_X)$   by $( \cdot \,, \cdot)_X$. The respective norms will be denoted by $\|\cdot\|_M$ and $\|\cdot\|_X$.   By Corollary \ref{cor21} to the generalized Weyl Criterion, for any $\lambda\in \sigma(k,\mathcal{L}^1_X), \lambda>0,$ there exists a sequence of compactly supported $k$-forms $\{\psi_j\}_{j\in\mathbb{N}}$ with $\|\psi_j\|_X=1$ such that for $m=1,2$
\begin{equation} \label{thm7e2}
|\,(\,(\mathcal{L}^1_X+1)^{-m}\psi_j,(\mathcal{L}^1_X-\lambda) \psi_j)_X \,|\to 0.
\end{equation}

By considering a subsequence of $\psi_j$ if necessary, we may in fact assume that the support of $\psi_j$ lies in $B_p(R_j)\subset X$. We let $\w_j=(f_j)^*(\psi_j)$, which are $k$-forms  on $ B_{x_j}(R_j)$ with compact support. Similar to the proofs of Theorem~\ref{PropBd} and Lemma \ref{resBd} and using assumption \eref{thm7e1}, we have
\begin{equation*}
\begin{split}
\bigl| (\, (\mathcal{L}_M^1+1)^{-m} (f_j)^*\psi_j,(\mathcal{L}_M^1-\lambda)(f_j)^* \psi_j\,)_M - &(\,(\mathcal{L}^1_X+1)^{-m} \psi_j,(\mathcal{L}^1_X-\lambda) \psi_j\,)_X \bigr| \\
&\leq C\eps_j\|(f_j)^*\psi_j\|_M^2
\end{split}
\end{equation*}
for some constant $C$ independent of $j$.
We then obtain
\begin{equation*}
|\,(\,(\mathcal{L}_M^1+1)^{-m}(f_j)^*\psi_j,(\mathcal{L}_M^1-\lambda) (f_j)^*\psi_j)_M \,|\leq 2C\eps_j \|(f_j)^*\psi_j\|_M^2
\end{equation*}
for all $j$ large enough and $m=1,2$.  Applying Corollary \ref{cor21},  we get $\lambda \in \sigma_\mathrm{ess}(k,\mathcal{L}_M^1)$.

If $0\in \sigma(k,\Delta_X)$, then we can prove directly that $0\in \sigma_\mathrm{ess}(k,\Delta_M)$. This completes the proof.
\end{proof}

\section{The Collapsing Case: Smooth Limit}\label{8}

We now consider Case {\it (iii)} of Theorem \ref{mainthm}.  Let $\{(x_i,R_i)\}$ be a minimal sequence. Let $M_i=B_{x_i}(R_i)$  and denote by $g_i = g\big|_{M_i}$ the restriction of the Riemannian metric to each ball.    Let $\mathcal R_{g_i}$ be the curvature tensor of $g_i$. Assume that the injectivity radius of $M_i$ at $x_i$ tends to zero and the sequence $(M_i,x_i)$  converges to a pointed metric space $(X,p)$.

For precision, we will denote by $(M,x,g,R)$ the geodesic ball of radius $R$ at the point $x$ of the manifold $M$ which is endowed with the metric $g$. To simplify notation we will sometimes use $M$, $(M,x)$, $B_x(R)$, or $M(R)$ to denote the same ball if there is no confusion.

As we have mentioned in the proof of Lemma \ref{Linj}, the curvature assumption on $M_i$ implies that the exponential map at $x_i$
\[
\exp_{x_i}: T_{x_i}M \to M_i
\]
is nondegenerate on vectors of norm at most $R_i$.  We also use $g_i$ to denote the pull back $\exp^*_{x_i} g_i$ of the metric to the tangent space $T_{x_i}(M)$. The pointed manifolds $$( T_{x_i}M,0, g_i, R_i)$$  now have injectivity radius uniformly bounded below.
Let
\[
\Gamma_i = \pi_1(M_i,x_i,g_i,R_i)
\]
denote the local pseudofundamental group. This is defined in \cite{Fuk1} as
\begin{equation*}
\begin{split}
\pi_1 (M_i,x_i,g_i,R_i) & =  \{\gamma: ( T_{x_i}M,0,  g_i, R_i) \to   ( T_{x_i}M,0, g_i, 2R_i) \; \big| \  \gamma \ \text{ is an isometric} \\&  \text{embedding with} \
  \gamma(0) \in ( T_{x_i}M,0,  g_i, R_i),  \ \  \text{and} \ \ \exp_{x_i} \circ \gamma =  \exp_{x_i} \}.
\end{split}
\end{equation*}

By Theorem 6.12 and Sublemma 8.10 in \cite{Fuk1}, a subsequence of $\left(\,( T_{x_i}M,0, g_i, R_i),\Gamma_i \right)$ converges in the equivariant pointed Gromov-Hausdorff topology to $$\left( \,(\mathbb{R}^n, 0, g_E, \infty), G \right),$$  where $G$ is a closed subgroup of $E(n)$, the isometry group of $\mathbb{R}^n$. By definition  there exist equivariant Hausdorff approximations
\begin{equation}\label{collapsing-2}
\tilde f_i:   ( T_{x_i}M,0, g_i, R_i)  \to (\mathbb{R}^n, 0, g_E, \infty)  \quad  \mathrm{and}\quad
h_i:   \Gamma_i  \to G
\end{equation}
such that for any $\eps>0$, and taking $i$ large enough, we get
\[
d(h_i (\gamma) (\tilde f_i (x)), \tilde f_i(\gamma ( x))) < \eps
\]
for all $x \in  T_{x_i}M_i (R_i), \gamma \in \Gamma_i$ satisfying $d(x,x_i), d(\gamma (x), x_i) < 1/\eps$ (see Definition 6.11 in \cite{Fuk1}).
The $h_i$ are not necessarily homomorphisms.

By Lemma 7.9 in \cite{Fuk1} ${(T_{x_i}M(2R_i),0)}/{\Gamma_i}$ is isometric to $B_{x_i}(R_i)$ in $M$. Using Lemma 6.13 of the same paper, we get

\begin{prop} \label{prop3_1}
Let $X =\mathbb{R}^n/G$. The pointed balls $(M_i,x_i,g_i,R_i)$ collapse in the pointed Gromov-Hausdorff distance to $(X, 0 , g_\infty, \infty)$, where $G$ is a closed Lie subgroup of $E(n)$.
\end{prop}

Consider a sequence $\eps_i \to 0$ such that $\sqrt{\eps_i} R_i\to \infty $ and  $\eps_i^{-1} \mathcal R_{g_i} \to 0$. When we rescale the metric of $M$ to $\tilde g_i=  \eps_i g_i$  the  ball $(M_i, x_i, g_i, R_i)$ gets mapped to $(M_i,x_i,\eps_i g_i, \sqrt{\eps_i}  R_i)$. We denote the rescaled ball
\[
\tilde M_i =(M_i, x_i, \eps_i g_i, \sqrt{\eps_i}R_i).
\]
The curvature of $\tilde M_i$ is given by $\eps_i^{-1} \mathcal R_{g_i}$, and therefore still vanishes as $i\to \infty$. Similarly, a subsequence of  $(M_i, x_i, \eps_i g_i, \sqrt{\eps_i}R_i)$ collapses to a metric space $(X', x_\infty' , g'_\infty, \infty)$.

Let
\[
\sigma_i: (M_i,x_i,g_i,R_i) \to (M_i,x_i,\eps_i g_i,\sqrt{\eps_i} R_i)
\]
denote the diffeomorphism which corresponds to the above rescale of the metric. Then we have the commutative diagram
\begin{equation}\label{d-1}
\begin{CD}
(M_i,x_i,g_i, R_i)   @>\sigma_i>>   ( M_i ,x_i,\eps_ig_i, \sqrt{\eps_i} \,R_i ) \\
@VVV @VVV\\
(X,x_\infty, g_\infty, \infty) @>>>  (X',x'_\infty,g'_\infty,\infty)
\end{CD}.
\end{equation}
The following result is well-known
\begin{lem}
The blow down of the metric space $X$ exists, and $X'$ is one of its possible blow downs.
\end{lem}

To obtain Theorem \ref{mainthm} in the collapsing case it will be necessary to modify the minimal subsequence. Lemma~\ref{lemSp2} will play an essential role to this purpose.

We take $R_i'=\delta\eps_i^{-1/2} (< R_i)$ for some positive number $\delta>0$. By Lemma~\ref{lemSp2}, there exist $y_i\in M_i$ such that $\{(y_i,R_i')\}$ is also a minimal sequence. To simplify notation we will still  denote $y_i$ by $x_i$. Then ~\eqref{d-1} becomes
\begin{equation}\label{d-2}
\begin{CD}
(M_i,x_i,g_i, \delta\eps_i^{-1/2})   @>\sigma_i>>   ( M_i ,x_i,\eps_ig_i, \delta) \\
@VVV @VVV\\
(X,x_\infty, g_\infty, \infty) @>>>  (X',x'_\infty,g'_\infty,\delta )
\end{CD}.
\end{equation}

The structure of the blow down space $X'$ will be critical in the computation of the spectrum of $M$. We will refer to it as a cone at infinity, in accordance with most literature.
\begin{Def} \label{cone}
Fix $\delta>0$. The metric space $(X',x_\infty', g_\infty', \delta)$ is a cone at infinity for the minimal sequence $(M_i,x_i, g_i, R_i)$ if there exists a sequence $\eps_i\to 0$  such that $\eps_i^{-1}\mathcal{R}_{g_i} \to 0$ and the sequence  $(M_i,x_i, \eps_i g_i, \delta)$  converges in the pointed Gromov Haussdorff sense to $(X',x_\infty', g_\infty', \delta).$
\end{Def}

In this section we assume that  the pointed metric space $(X',x_{\infty}', g_\infty',\delta)$ is smooth. Moreover, by shrinking $\delta$ if necessary, we may assume that the injectivity radius of $x'_\infty$ at $X'$ is bigger than $\delta$. In particular, $(X',x_{\infty}', g_\infty',\delta)$ is diffeomorphic to a small  open ball of $\mathbb R^n$.\footnote{Here we need to use Lemma~\ref{lemSp2}  again.}

By the smoothing  theorems of Cheeger-Gromov and Abrech,  we may assume that the metrics  $\eps_ig_i$  on $(M_i,x'_i,\eps_ig_i, \delta)$  and  $g'_\infty$  on $(X',x_\infty',g'_\infty,\delta)$ are $A$-regular as in \cite{CFG}*{Theorem 1.12}.  In particular, the higher order derivatives of the curvature of $(X',x_\infty',g'_\infty,\delta)$ are bounded. As a result,  the curvature of $X'$ is also bounded.

Let
\[
f_i: (M_i,x_i, \eps_i g_i, \delta) \to (X', x_\infty' , g'_\infty, \delta)
\]
be the smoothing fibrations defined in~\cite{CFG}*{Theorem 2.6}. We can in fact require that the functions $f_i$ satisfy (2.6.1) through (2.6.7) in~\cite{CFG}*{page 338}. In particular,
\begin{equation}\label{covariantdev}
|\n^k f_i | \leq C(k)
\end{equation}
for any $k$, independent of $i$.

Define the maps
\begin{equation}\label{19-1}
F_i: (T_{x_i}M_i,0, \eps_i g_i, \delta) \to (X', x_\infty' , g_\infty', \delta)
\end{equation}
via composition with the  exponential map   $\pi=\exp_{x_i} : T_{x_i}M_i \to M_i$. Note that $F_i(x)=\tilde f_i(\eps_i\,x)$, where the $\tilde f_i$ are defined in~\eqref{collapsing-2}.

\begin{remark} Even though the curvatures of $T_{x_i}M_i$ tend to zero, the limit space $X'$ need not be flat.  This can be seen in the following example.  Let $\mathbb{R}^2 \times \mathbb{R}^1$ be the flat 3-dimensional Euclidean space, which we identify with $\mathbb{C}^1\times \mathbb{R}^1$. Consider the $\mathbb{Z}$-action on $\mathbb{C}^1\times \mathbb{R}^1$ given by
\[
(z,x) \mapsto (e^{i n} z,   n + x)
\]
for $n \in \mathbb{Z}$. Let
\[
M_\eps= {{\mathbb{C}^1\times \mathbb{R}^1}/{\eps \mathbb{Z}}  }.
\]
Then $M_\eps$ converges in the Gromov-Hausdorff sense to ${{\mathbb{C}^1\times \mathbb{R}^1}/{ \mathbb{R}}   } = \mathbb{R}^2$ endowed with the metric
\[
g= d r^2 + \frac{r^2}{r^2 +1} \; d\theta^2
\]
written in polar coordinates, which has strictly positive curvature.

\end{remark}

We will now prove that the spectrum on the minimal sequence can be computed as if these balls were the product of Euclidean balls with almost flat manifolds. Although the proof is quite technical, the idea itself is simple. By the classical O'Neill Theorem for a Riemannian submersion, the curvatures of the two Riemannian metrics, the second fundamental form of the fibers, and the curvature of the horizontal distribution of the fibrations are all related. In particular, whenever three of the above quantities are small, then so is the fourth one. In our setting, by Cheeger-Fukaya-Gromov theory the fibrations are almost submersive. Moreover, after rescaling the metrics $\eps_i g_i$ back to $g_i$, both the curvature of $M_i$ and that of $X'$ are sufficiently small. As a result, the curvature of the horizontal distribution of the fibration is also small, proving that the fibrations are almost product.

To simplify notation we will suppress $i$ below.
Fix a point $p\in  M(\delta)$ and let $p_\infty=F(p)$.  Let $q=\dim X'$. Note that $q<n$. We let $(x_1, \cdots , x_n)$ and $(y_1, \cdots, y_q)$ be harmonic coordinate systems at  $p$ and $p_\infty$ respectively. Without loss of generality, we assume that $p=p_\infty=0$. Denote by $D^k$ the ordinary $k$th order derivatives with respect to these coordinates. Then from ~\eqref{covariantdev} and the harmonicity of the coordinate systems, we have
\[
|D^l F | \leq \tilde C(l),
\]
where $\tilde C(l)$ are constants depending only on $l$.

By using a linear map at $T_x  M$ we may assume that
\begin{equation}\label{17}
\begin{split}
&\frac{\p y_\beta}{\p x_\alpha} (0) = \delta_{\alpha \beta} \ \ \text{for} \ \ 1\leq \alpha, \beta \leq q \\
&\frac{\p y_\beta}{\p x_\alpha} (0) = 0 \ \ \text{for} \ \  \alpha > q.
\end{split}
\end{equation}

By the implicit function theorem, there exists a neighborhood $\mathcal{O}$ of $0\in T_x  M$  on which the coordinates $x_{q+1}, \cdots, x_n$ are local coordinates of the fiber $F^{-1}(0)$. Moreover, since the derivatives of $F$ are bounded then, by shrinking  $\delta$ if necessary, we have
\begin{enumerate}
\item for each $t\in X'(\delta)$, $(x_{q+1},\cdots,x_n)$ is a local coordinate system of the fiber $F^{-1}(t)$ near $p=0$.  Here we use $X'(\delta)$ to denote $(X', x_\infty', g_\infty',\delta)$ for short;
\item On each fiber $F^{-1}(t)$, the derivatives of $x_1,\cdots,x_q$ with respect to $x_{q+1},\cdots, x_n$ are bounded independently of $i$.
\end{enumerate}

Let $V= \sum_{\beta =1}^q b_\beta \frac{\p}{\p y_\beta}$ be a vector field on $X'(\delta)\subset X'$. Then there exists a unique lift $\tilde V$ on both $(M,x,\eps g, \delta)$ and $(T_x M, 0,\eps g, \delta)$  such that
\begin{equation}\label{lift}
\begin{split}
& dF(\tilde V) =V \ \ \text{and} \\
& \tilde V \ \ \text{is normal to the fibers}.
\end{split}
\end{equation}
To simplify notation, we do not distinguish the vector fields on $(M,x,\eps g, \delta)$ and $(T_x M, 0,\eps g, \delta)$.

For $\alpha >q$ the vector field $\p/\p x_\alpha$, in the coordinate system $(x_{q+1},\cdots, x_n)$,  is tangent to the fiber. As a vector field on $T_{x}M$,  it is represented by
\[
\frac{\p}{\p x_\alpha} = \frac{\p}{\p x_\alpha} + \sum_{\beta \leq q} \frac{\p x_\beta}{\p x_\alpha} \frac{\p}{\p x_\beta} := \sum_{\beta =1}^n C_{\alpha,\beta} \frac{\p}{\p x_\beta}.
\]

Let $\tilde V= \sum_{\alpha =1}^n a_\alpha \frac{\p}{\p x_\alpha}$ be the lift. Then by~\eqref{lift}, $\tilde V$ satisfies
\begin{equation*}
\begin{split}
& \sum_{\alpha=1}^n a_\alpha \frac{\p y_\beta}{\p x_\alpha}= b_\beta  \ \ \text{for $\beta\leq q$, and } \\
&  \sum_{\beta,\nu=1}^na_\nu C_{\alpha,\beta} \; g_{\beta \nu} =0 \ \ \text{for} \ \  \alpha > q.
\end{split}
\end{equation*}

We consider the $n\times n$ matrix $(\xi_{\alpha \beta})$ defined by
\begin{equation*}
\xi_{\alpha \beta} = \left\{
\begin{array}{cc}
   \frac{\p y_\alpha}{\p x_\beta} &  \alpha\leq q, \ \beta\leq n \\
\sum_{\nu=1}^n C_{\alpha,\nu} \; g_{\nu \beta} &  \alpha>q, \beta \leq n
\end{array}\right..
\end{equation*}

Then $\xi_{\alpha \beta}(0)$ is nonsingular. By the implicit function theorem (and by shrinking $\delta$ if necessary) the $a_\alpha$, or in other words the lift vector field $\tilde V$, are uniquely determined by  $F$. Moreover, all the derivatives of $\tilde V$ are uniformly bounded.

$X'(\delta)$ is diffeomorphic to a ball of radius $\delta$ in $\R^q$ via the exponential map. Let $\vec b=(b_1,\cdots,b_q)\in\mathbb R^q$ be a vector such that $\|\vec b\|<\delta$. We identify $\vec b$ with the vector field
\[
\sum_{j=1}^q b_j\frac{\pa}{\pa y_j},
\]
on $X'(\delta)$. The unique lift of $\vec b$ is denoted by $\tilde V_{\vec b}^i$ or $\tilde V_{\vec b}$.

Now consider the flow $\tilde \sigma$  defined by the vector field $V_{\vec b}$ on $M$ such that
\[
\frac{d \tilde \sigma}{dt} = \tilde V_{\vec b},  \quad \tilde\sigma(0)\in F^{-1}(0)
\]
and the corresponding flow on $X'$
\[
\frac{d \sigma}{dt} =  \vec b, \quad \sigma(0)=0.
\]
By definition, we have $F\circ\tilde \sigma=\sigma$.
\begin{lem}
By further shrinking $\delta$ if necessary, the flow provides the diffeomorphism
\begin{equation*}
\begin{split}
\nu: F^{-1}(0) \times X'(\delta) & \to F^{-1} (X'(\delta)) \quad  \\
(z,\vec b) & \mapsto \tilde \sigma(1)
\end{split}
\end{equation*}
where $\tilde\sigma$ is the flow of $\tilde V_{\vec b}$  with initial value $\tilde\sigma(0)=z$.
\end{lem}
\begin{proof} Given that all the derivatives of $F$ and $\tilde V_{\vec b}$ are bounded (and the map is not degenerate),  we conclude that the map is a diffeomorphism for $\delta$ sufficiently small.
\end{proof}

$F^{-1}(0) \times X'(\delta)$ carries a natural Riemannian product metric. We shall compare this metric to the metric on $F^{-1} (X'(\delta))$ via the diffeomorphism defined in the above lemma.   $F$ is an ``almost'' submersion (as in \cite{CFG}*{page 337}) in the sense that
\[
\|\vec b\|=\|dF(\tilde V_{\vec b})\|\approx \|\tilde V_{\vec b}\|,
\]
and the fiber is ``almost'' totally geodesic. However, these properties are not sufficient to make the metric on
$F^{-1} (X'(\delta))$   ``close'' to the product metric. According to O'Neill's theorem, a submersion with totally geodesic fibers does {\it not} necessarily correspond to a product structure on the manifold, and this obstruction is reflected in the $A$-tensor as defined in \cite{Oneill}. Given the above obstruction, we will use a rescaling argument to find a nearby flat metric to $F^{-1} (X'(\delta))$.

We scale back and reintroduce the subscript $i$. Endow $F^{-1}_{i}(0)$ with the metric that is given by the restriction of $g_i$. Endow $X'(\delta)$ with the metric $\eps_i^{-1} g'_\infty$ and  $F^{-1}_{i}(0) \times X'(\delta)$ with the product metric $\eps_i^{-1} (g_i\times g'_\infty)$.

$X'(\delta)=(X',\infty',g_\infty',\delta)$ is rescaled to the manifold $(X',x'_\infty,\eps_i^{-1} g_\infty', \delta\eps_i^{-1/2})$.

We will study $V_{\vec b}^i$ as $i\to \infty$ on the space $(M_i,x_i,g_i,\delta\eps_i^{-1/2})$.
First, we rescale the coordinate systems to
\begin{equation*}
\begin{split}
&  (x_1, \cdots, x_n) \to (\eps_i^{-1} x_1,  \cdots, \eps_i^{-1} x_n)  \ \ \text{and } \\
&   (y_1, \cdots, y_q) \to (\eps_i^{-1} y_1,  \cdots, \eps_i^{-1} y_q).
\end{split}
\end{equation*}
We also use $(x_1,\cdots, x_n)$ and $(y_1,\cdots,y_q)$ for the new coordinates to simplify notation. We observe that the functions $F_i$ under these two coordinate systems are scaled in the following way
\begin{enumerate}
\item the first order derivatives are not scaled;
\item all of the higher order derivatives tend to zero.
\end{enumerate}
By the above observation, using ~\eqref{17}, we have
\[
\xi_{\alpha \beta} \to I \ \ \text{as} \ \ i\to \infty.
\]

Let $\hat g_i$ be the product metric induced by
$g_i$ on $F_i^{-1}(0)$ and $\eps_i^{-1} g_\infty'$.
The key lemma of  this section is the following
\begin{lem} For any $\eps>0$ and any $\rho\gg 0$, we have
\[
|\nu^*(g_i)/\hat g_i-1|_{\hat g_i}<\eps
\]
on $(M_i,x_i,g_i,\rho)$.
\end{lem}

\begin{proof}
Both $g_i$ and $\hat g_i$ are invariant under the local pseudo fundamental group. Therefore we can compare them over the tangent spaces $T_{x_i}M_i$.
It is obvious that as $i\to\infty$ both limits
\[
g_\infty=\lim_{i\to\infty} g_i, \quad \hat g_\infty=\lim_{i\to\infty} \hat g_i
\]
exist and are flat.  What we need to prove is that
\[
g_\infty=\hat g_\infty.
\]

First, by restricting to $F^{-1}_i(0)$, both $g_\infty$ and $\hat g_\infty$ are the same. It is also not difficult to see that $V_{\vec b}^\infty$ is orthogonal to the fibers on both metrics. Second, by~\cite{CFG}, the maps $F_i$ are ``almost'' submersions which implies that the limit $F_\infty$ is a submersion. Therefore, the norm of $V_{\vec b}^\infty$ is the same with respect to either metric. This concludes the proof.
\end{proof}

In summary, we proved the following
\begin{thm}\label{thm10}
Assume that for some sequence $\eps_i\to 0$  such that $\eps_i^{-1}\mathcal{R}_{g_i} \to 0$, and a fixed $\delta>0$  the limit  $(X',x_\infty', g_\infty', \delta)$ of the sequence $(M_i,x_i, \eps_i g_i, \delta)$ is a  manifold. Then for any $x\in (X',x_\infty', g_\infty', \delta)$, and some  $\delta>0$ the fibration
\[
f_i: (M_i, y_i, g_i,\delta\eps_i^{-1/2})\to (X',x, g_\infty',\delta\eps_i^{-1/2})
\]
is almost product, where $f_i(y_i)=x$.
\end{thm}

When the limit space $X'$ of the rescaled sequence $\tilde M_i$ is smooth then for any $x\in X'$ $f_i^{-1}(x)$ is an almost flat manifold of dimension $n-s$ \cite{CFG}*{Theorem 2.6}\footnote{Possibly, $s$ depends on $i$, but  we can always choose a subsequence.}. We endow  $f_i^{-1}(x)$ with a metric that is given by the restriction of $g_i$ to this fiber. For any $0\leq k\leq n$, define
\[
\lambda_{\rm exp}= \liminf_{i\to\infty}\inf_{x\in X'}
\alpha(f_i^{-1}(x),s,n,k)
\]
to be the \emph{expected} infinium of the essential spectrum, where $\alpha(f_i^{-1}(x),s,n,k)$ is as in Definition~\ref{lem81}. We will show the following
\begin{thm}\label{thm8b} Using the above notations, for any $0\leq k\leq n$, we have
\[
\lambda_{\rm exp}=\lambda^{{\rm ess}}_o(k,\Delta,M).
\]
\end{thm}
\begin{remark} The definition of $\lambda_{\rm exp}$ depends on the approximation functions $f_i$, therefore  it is not intrinsically defined. The above theorem however, shows us that it is indeed intrinsic, because it is equal to the infimum of the essential spectrum.
\end{remark}

\begin{proof}[Proof of Theorem~\ref{thm8b}]
By Lemma~\ref{lemSp2}, for any $R\gg 0$, there exists a point $y_i\in ( M_i ,x_i, g_i, \delta\eps_i^{-1/2} )$ such that the balls $\bar M_i=( M_i ,y_i, g_i, R )$ satisfy
\[
\lambda_o(k,\Delta_i,\bar M_i)\leq \lambda_o^{\rm ess}(k,\Delta,M)+O(R^{-2})
\]
as $R\to\infty$, where $\delta$ is as in Theorem \ref{thm10}.  This implies that
\[
\lambda_{\rm exp} \leq \lambda^{{\rm ess}}_o(k,\Delta,M).
\]

To prove the  reverse inequality, we choose $x\in X'$. Then in the $\delta$ neighborhood of $x$, for $i\gg 0$, the fibration $f_i$ is almost product in the sense of Theorem~\ref{thm10}. It is thus clear that
\[
\lambda_{\rm exp} \geq\lambda_o^{\mathrm{ess}}(k,\Delta,M)+o(1).
\]
The theorem is proved.
\end{proof}

\section{The Collapsing Case: Singular Limit}\label{9}
The complexity of computing the spectrum in the collapsing case lies in the fact that the singular set of the cone at infinity $X'$ of the rescaled sequence $\tilde M_i$  can be fairly complicated. In this section we will study two interesting cases for which the essential spectrum is computable. In the first case we will suppose that the singular locus of the cone at infinity $X'$ consists of an isolated point. We will prove the following result.
\begin{thm} \label{thm9c}
Let $M$ as in Theorem \ref{mainthm}. Assume that the minimal sequence of geodesic balls is collapsing and that for some sequence $\eps_i\to 0$  as in Definition \ref{cone} the cone at infinity $X'$  has a single singularity point.  Then $\sigma_\mathrm{ess}(k,\Delta_M)$ is either empty, or a connected interval.
\end{thm}

\begin{proof}
Without loss of generality we assume that the singular point is $x'_\infty$, the limit of the $x_i$.
Otherwise one can shrink the sequence a little bit.

Consider the  $k$-eigenforms $\omega_i$ on $\overline M_i =( M_i ,x_i, g_i, \delta\eps_i^{-1/2} )$ such that
\[
\Delta_i \omega_i = \lambda_i \omega_i,\quad |\omega_i||_{\pa \overline M_i}=0,
\]
where $\lambda_i\geq 0$ is the first eigenvalue of the Laplacian $\Delta_i$ on  $\overline M_i$.

Let $\tilde \omega_i$ be the lift of $\omega_i$ to  the tangent space $(T_{x_i}\overline M_i, 0,  g_i, \delta\eps_i^{-1/2})$ given via the covering map
\[
(T_{x_i}\bar M_i, 0,  g_i, \delta\eps_i^{-1/2}) \to  \bar M_i
\]
and chose $\w_i$ such that  $\tilde \w_i$ satisfies $\|\tilde \w_i\|_{L^2(T_{x_i}\bar M_i(4))}=1$.

By~\eqref{collapsing-2}, we  have the Gromov-Hausdorff convergence
\[
\tilde f_i: (T_{x_i} \bar M_i, 0,  g_i, R) \to (\mathbb{R}^n, 0, g_E, R)
\]
where $g_E$ is the flat metric on $\mathbb R^n$. Since we assume that the metrics $g_i$ are $A$-regular in the sense of ~\cite{CFG}, the above convergence is  in the $\mathcal C^\infty$-sense.

By passing to a subsequence if necessary, we have that
\begin{equation*}
\tilde \omega_i  \to \tilde \omega_\infty
\end{equation*}
in the $\mathcal C^\infty$-sense. Since $\|\tilde \w_i\|_{T_{x_i}\bar M_i(4)}=1$, we have  $\|\tilde \w_\infty\|_{B(0,4)} =1$, where $B(0,4)$ is the Euclidean ball of radius 4 centered at $0$. Therefore, the limit form $\tilde \w_\infty$ is nonzero. Moreover, $\|\tilde \w_\infty\|_{\mathbb R^n} = \infty$ since $\tilde \w_\infty$ is a nonzero eigenfunction on  $\mathbb R^n$. This implies that
\[
\lim_{i\to \infty} \|\tilde \w_i\|_{T_{x_i}\bar M_i(\delta\eps_i^{-1/2})} =\infty.
\]

As a result, for each $m\in \mathbb{N}$ there exists an $i(m)\to \infty$ such that
\[
\|\tilde \w_i\|_{T_{x_i}\bar M_i(\delta\eps_i^{-1/2})} \geq 2 \|\tilde \w_i\|_{T_{x_i}\bar M_i(m)}.
\]
Let $\Gamma_i$ denote the local pseudofundamental group as in the previous section.  By taking the quotient with the discrete group $\Gamma_i$ we have
\begin{equation} \label{e9_4}
\|\w_i\|_{\bar M_i } \geq 2 \|  \w_i\|_{B_{x_i}(m)}.
\end{equation}

Let $m$ and $i=i(m)$ as above. As in Section \ref{s7} we construct a Gromov cover on $\bar M_i$ but in a slightly different manner; the Gromov cover includes the ball of radius of radius $m$ at $x_i$, $B_{x_i}(m)$, and a finite number of balls  $\{B_{y_j}(m/8)\}_{j=1}^N$ that cover  $\bar M_i\setminus B_{x_i}(m/2)$, such that $y_i \in \bar M_i\setminus B_{x_i}(m/2)$  and each point in $\bar M_i\setminus B_{x_i}(3m/4)$ is covered at most $C(n)$ times. For each $j=1,\ldots, N$ we let $\phi_j$ be the cut-off functions on $B_{y_j}(m/8)$ as in Section \ref{s7}, and  let $\phi_o$ be the cut-off function which is equal to $1$ on  $B_{x_i}(m/2)$ and vanishes outside $B_{x_i}(m)$.  Let $\rho_j$ as in \eref{puni} and assume that
\[
\|\n (\rho_j \w_i)\|^2 \geq \hat{\lambda} \|\rho_j \w_i\|^2
\]
for all  $j=1,\ldots, n$. Since
\[
\|\n (\rho_o \w_i)\|^2 \geq \lambda_o^{\mathrm{ess}} \|\rho_o \w_i\|^2
\]
we have
\begin{equation*}
\begin{split}
\hat{\lambda}  \sum_{j=1}^N  \|\rho_j \w_i\|^2 + \lambda_o^{\mathrm{ess}} \|\rho_o \w_i\|^2
&\leq \sum_{j=1}^N  \|\n (\rho_j \w_i)\|^2 + \|\n(\rho_o \w_i)\|^2 \\
&\leq \left( \frac{C(n)}{m^2} + \lambda_o^{\mathrm{ess}} + o(1) \right) \|\w_i\|^2
\end{split}
\end{equation*}
by equation \eref{2-est} and the fact that we have a minimal sequence.  Since
\[
\|\w_i\|^2 =  \sum_{j=1}^N  \|\rho_j \w_i\|^2 + \|\rho_o \w_i\|^2
\]
we get
\[
\hat{\lambda}  \sum_{j=1}^N  \|\rho_j \w_i\|^2 \leq  \left( \frac{C(n)}{m^2} + \lambda_o^{\mathrm{ess}} + o(1) \right) \sum_{j=1}^N  \|\rho_j \w_i\|^2 + \left( \frac{C(n)}{m^2} +  o(1) \right)  \|\rho_o \w_i\|^2.
\]

On the other hand, \eref{e9_4} implies that
\begin{equation*}
  \begin{split}
\|\rho_o \w_i\|^2& \leq \|\w_i\|_{B_{x_i}(m)} \leq \frac 12 \|\w_i\|_{\bar M_i} = \frac 12 \|\rho_o \w_i\|^2 +\frac 12  \sum_{j=1}^N  \|\rho_j \w_i\|^2 \\
\Rightarrow \|\rho_o \w_i\|^2& \leq \sum_{j=1}^N  \|\rho_j \w_i\|^2.
  \end{split}
\end{equation*}
Applying this estimate to the right side of the inequality above we have
\[
\hat{\lambda}  \sum_{j=1}^N  \|\rho_j \w_i\|^2 \leq  \left( \frac{C'(n)}{m^2} + \lambda_o^{\mathrm{ess}} + o(1) \right) \sum_{j=1}^N  \|\rho_j \w_i\|^2
\]
and we conclude that
\[
\hat{\lambda}   \leq   \frac{C'(n)}{m^2} + \lambda_o^{\mathrm{ess}} + o(1).
\]
In other words we can use the balls $B_{y_j}(m/8)$ to capture the bottom of the essential spectrum. Since  these balls converge to a smooth subset of $X'$, the theorem follows from the results of Section \ref{8}.
\end{proof}

With Theorem \ref{thm9c} we have completed the proof of Theorem \ref{mainthm}.

The second case we consider is when the limit $X$ of the unscaled sequence and a cone at infinity $X'$ have the same dimension and  the diameter of the fibers over $X$  uniformly collapses  as $i\to \infty$.
\begin{Def}
Let $( M_i ,x_i, g_i, \delta\eps_i^{-1/2} )$ be a minimal sequence which converges to $(X,x_\infty, g_\infty, \infty)$  and let $(X',x'_\infty,g'_\infty,\delta )$ be as in Definition \ref{cone}. We say that the minimal sequence {\it is strongly collapsing} if the limit spaces $X$ and $X'$ have the same dimension, the blow down of the singular set of $X$ coincides with the singular set of $X'$, and the diameter of the fibers over $X$  uniformly collapses  as $i\to \infty$.
\end{Def}
We will prove the following.
\begin{thm} \label{thm9b}
Let $M$ as in Theorem \ref{mainthm}. Assume that the minimal sequence of geodesic balls is strongly collapsing. Then $\sigma_\mathrm{ess}(k,\Delta_M)$ is either empty, or $[0,\infty)$.
\end{thm}

\begin{proof}

Let $f_i$ denote the $\eps$ approximation maps
\begin{equation}\label{fibration}
f_i: ( M_i ,x_i, \eps_i g_i, \delta ) \to (X',x_\infty', g_\infty',\delta)
\end{equation}
for some $\delta>0$, for $i$ large enough.

Let $X'_\mathrm{sing}$ be the singular locus of $X'$.  Let $X'_{\rm sing}(\eps)$ be the $\eps$-neighborhood of $X'_{\rm sing}$ and denote by $\pa X'(\eps)$ the $\eps$-neighborhood of $\pa X'$.  Since $X'_{\rm sing}$ consists of a  finite union of submanifolds, the measure  of  $X'_{\rm sing}(\eps)$ is no more than $O(\eps)$ \cite{Fuk3}. 

For $i\gg 0$, we consider the points  $x\notin X'_{\rm sing}(\eps)\cup \pa X'(\eps)$.

\begin{lem}
Assume that the dimensions of $X$ and $X'$ are the same. Let $x\in X'$ be a point, and let $Z_i(x)$ be the fiber with respect to the fibration~\eqref{fibration}. Then
\[
{\rm diam}\, (Z_i(x))/\sqrt{\eps_i}\to 0
\]
as $i\to\infty$.
\end{lem}

\begin{proof}
If not, then after rescaling back, if there is a $\delta$ such that
\[
{\rm diam}\, (Z_i(x))/\sqrt{\eps_i}> \delta
\]
for $i\gg 0$. Thus the dimension of the limit is greater than the dimension of $X'$, which is a contradiction.
\end{proof}

 As we have seen in the previous section $f_i^{-1}(x)$ is an almost flat manifold of dimension $n-s$.

Using the same notation as \cite{lott}, we note that the fiber $Z$ can be given the structure of an infranil manifold $Z=N/\Gamma$ where $\Gamma$ is a discrete subgroup of $\mathrm{Aff}(N)$. Then the flat connection $\n^{\mathrm{aff}}$ on $N$ descends to a flat coneection on $TZ$. We denote $\Delta^{Z}$  the Laplacian on forms of this fiber and $\Delta^{\mathrm{inv}}$ the Laplacian on the space of differential forms on $Z$ which are parallel with respect to $\n^{\mathrm{aff}}$.

Let $\lambda_i(f_i^{-1}(x),k)$ denote the first eigenvalue of $\Delta^{Z}$ on $k$-forms over the fiber $Z=f_i^{-1}(x)$. We have the following result, whose proof we defer to the end of this section.
\begin{lem} \label{lem9a}
Suppose that for a point $x\in X'-X'_{\rm sing}$, the fiber $Z=f_i^{-1}(x)$ has a nonzero parallel $l$-form with respect to the affine connection for some $l\leq n/2$. Then
\[
\sigma_{\rm ess}(p,\Delta_M)=[0,\infty)
\]
for $l\leq p\leq n-l$ if $l+s\geq n/2$, and for $l\leq p\leq l+ s$ and $n-(s+l) \leq p \leq n-l$ otherwise.
\end{lem}

Consider  $k\leq n/2$. If $k\leq s,$ then it is clear from the previous section that $\sigma_\mathrm{ess}(k,\Delta_M)=[0,\infty)$ by constructing approximating eigenfunctions using the smooth part of $X'$.

If $s<k \leq n/2$  we have two cases. If the fiber $Z=f_i^{-1}(x)$ has a nonzero parallel $l$-form for some $l\geq 1$, then by Lemma \ref{lem9a}  $\sigma_\mathrm{ess}(k,\Delta_M)=[0,\infty)$ for $l\leq p\leq n-l$ if $l+s\geq n/2$, and  for  $l\leq p\leq l+ s$ and $n-(s+l) \leq p \leq n-l$ otherwise.

To complete the proof of the theorem it suffices to show that in the remaining case the essential spectrum is empty.

Consider the  $k$-eigenforms $\omega_i$ on  $M_i= ( M_i ,x_i,  g_i, \delta \eps_i^{-1/2} )$ such that
\[
\Delta_i \omega_i = \lambda_i \omega_i,\quad |\omega_i||_{\pa M_i}=0,
\]
where $\lambda_i\geq 0$ is the first eigenvalue of the Laplacian $\Delta_i$ on  $M_i$.

Let $R$ be a large positive number. We rescale $\w_i$ such that $\|\w_i\|^2_{L^2(M_i(R))} / \mathrm{Vol}(M_i(R)) =1$ where $M_i(R)= ( M_i ,x_i,  g_i, R )$. For simplicity we will take $R=1.$

Suppose that the eigevalues $\lambda_i$ are uniformly bounded. Using the asymptotical flatness of the geodesic balls (and their tangent spaces), the Poincar\'e inequality and a Moser's iteration argument give us
\begin{equation} \label{e9_3}
\|\w_i\|_{L^{\infty}(M_i(1))} \leq C
\end{equation}
for some uniform constant $C$ independent of $i$.

For any $r>0$ define
\[
Y_i(r)=\{ y \in  M_i(1)\, \big| \, \mathrm{dist} (f_i(y), X_{\rm sing}) \geq r \}.
\]

Fix $\eps_o>$ and  consider $Y_i(\eps_o).$  We claim that there exists a positive constant $C$ such that for all $i\gg 1$ there exists a point $y_i \in Y_i$ such that
\[
\frac{1}{\mathrm{Vol}(M_i(1))} \int_{B_{y_i}(\eps_o/2)} |\w_i|^2 \geq C_o.
\]

If not, then for any $m\in \mathbb{N}$ there exists an $i=i(m)$ such that for all $y\in Y_i(\eps_o)$
\[
\frac{1}{\mathrm{Vol}(M_i(1))} \int_{B_{y}(\eps_o/2)} |\w_i|^2 \geq \frac 1m.
\]

Taking an $\eps/2$ Gromov cover of $Y_i(\eps_o)$ such that at most $C(n)$  balls of radius $\eps/2$ intersect at each point, we have that
\[
\frac{1}{\mathrm{Vol}(M_i(1))} \int_{Y_i(\eps_o/2)} |\w_i|^2 \geq \frac{C(n)}{m}.
\]
This implies that
\[
\frac{1}{\mathrm{Vol}(M_i(1))} \int_{M_i(1)\setminus   Y_i(\eps_o/2)} |\w_i|^2 \geq 1- \frac{C(n)}{m} \geq \frac 12
\]
for all $i$ large enough. Combining this with the uniform bound \eref{e9_3} we get
\[
\frac{\mathrm{Vol}(M_i(1)\setminus   Y_i(\eps_o/2))}{\mathrm{Vol}(M_i(1))} \geq \frac C2
\]
independently of $\eps_o$. But this contradicts the fact that the volume of $M_i(1)\setminus   Y_i(\eps_o/2)$ should vanish as $\eps_o \to 0$.

As a consequence of the claim we now have
\begin{equation} \label{e9_2}
\frac{\int_{M_i(1)}|\n  \w_i|^2}{\int_{M_i(1)}|\w_i|^2} \geq C_o \frac{\int_{B_{y_i}(\eps_o/2)}|\n  \w_i|^2}{\int_{B_{y_i}(\eps_o/2)}|\w_i|^2}.
\end{equation}

By Proposition 2 in \cite{lott}
\[
\sigma(l,\Delta^{Z}) \cap [0, A \cdot \text{diam}(Z)^{-2}) =\sigma (l,\Delta^{\mathrm{inv}}) \cap  [0, A \cdot \mathrm{diam}(Z)^{-2})
\]
for some constant $A$ that depends only on the dimension of $Z$. The order of the $k$-form requires that the form must have a component along the fiber. However, the absence of parallel $l$-forms on the fibers $Z$ implies that
\[
\frac{\int_{{B_{y_i}(\eps_o/2)}} |\n    \w_i|^2}{ \int_{{B_{y_i}(\eps_o/2)}} | \w_i|^2} \geq A \cdot \text{diam}(Z)^{-2}.
\]

As a result, the right side of \eref{e9_2} becomes infinitely large as $i\to \infty$. In particular we get that for any $R\gg1$
\begin{equation*}
\frac{\int_{M_i(R)}|\n  \w_i|^2}{\int_{M_i(R)}|\w_i|^2} \to \infty
\end{equation*}
as $i\to\infty$. This however contradicts the uniform upper bound on the eigenvalues $\lambda_i$. Therefore $\lambda_i \to \infty$ and the essential spectrum must be empty.
\end{proof}

\begin{proof}[Proof of Lemma~\ref{lem9a}]
Denote $\|\cdot\|=\|\cdot\|_{L^2(Z)}$. By \cite{ruh} the covariant Laplacian $\n^Z$ and  $\n^{\mathrm{aff}}$ are close on almost-flat manifolds, and their difference is given by the curvature of the space. Since in our case the curvature of $Z$ approaches $0$ as $i\to \infty$ and its diameter is small, we have that $|\n^N - \n^{\mathrm{aff}}| = o(1)$.

Let $\eta$ denote a parallel $l$-eigenform for $\Delta^{\mathrm{inv}}$ over the fiber $Z$, such that $\n^{\mathrm{aff}} \eta =0$.  Then
\[
\frac{\|\n^Z \eta\|}{\|\eta\|} =o(1).
\]
The lemma follows by the asymptotic flatness of $Z$.
\end{proof}

\section{The $k$-form Spectrum over Manifolds with Asymptotically Nonnegative Ricci Curvature}\label{10}

In this final section we will prove Theorem \ref{thmRic}. For the positive sequence $R_i$ in the assumption of the theorem,
let $\eps_i\to 0$ such that $ \sqrt{\eps_i} R_i \geq 7$ and $\eps_i^{-1} \delta_i \to 0$. Consider the rescaled balls\footnote{To simplify notation, and if it is otherwise clear, we also denote by $g$ the restriction, $g_i$, of $g$ to $M_i$.}
\[
(M_i, x_i, \eps_i g, \sqrt{\eps_i} R_i).
\]
Then
\[
\mathrm{Ric}_{(M_i, x_i, \eps_i g, \sqrt{\eps_i} R_i)} \geq - \eps_i^{-1} \delta_i \to 0.
\]
By the Gromov compactness theorem, possibly after taking a subsequence,
\begin{equation}\label{conv2}
(M_i, x_i, \eps_i g, \sqrt{\eps_i} R_i) \to (X', x_\infty', g_\infty', R_o)
\end{equation}
under  the pointed Gromov-Hausdorff distance, where $5\leq R_o\leq \infty$ and $(X', x_\infty',g_\infty',R_o)$ is a pointed complete metric space.

\begin{lem} \label{lemGH}
Under the assumption of Theorem~\ref{thmRic},
there exists a sequence $p_i\in M_i$, and a positive sequence $r_i>0$ with $r_i^{-1}\eps_i\to 0$, such that
\[
d_{GH}\left((M_i, p_i, r_i^{-1}\eps_i g, 5),(\mathbb{R}^q, 0, g_E, 5)\right) \to 0
\]
as $i\to \infty$, for some $q>0$. Here $g_E$ is the standard Euclidean metric of $\R^q$.
\end{lem}

\begin{proof}
Let
\[
\mu_i=d_{GH}((M_i, x_i, \eps_i g, 7), (X',x'_\infty, g'_\infty, 7)).
\]
Then by~\eqref{conv2}, $\mu_i\to 0$.

Let $p\in X'$ be a $q$-regular ($q>0$) point defined in \cite{CCoIII}, which always exists under the assumption of Ricci curvature uniformly bounded below, and let $U$ be a small metric ball of radius $\delta$ at $p$. Since the tangent cone at the $q$-regular point $p$ is $\mathbb{R}^q$, for any sequence $r_i\to 0$,
\[
\sigma_i=d_{GH}((U,p,r_i^{-1} g'_\infty, \delta r_i^{-1/2}), (\R^q,0,g_E, \delta r_i^{-1/2}))\to 0,
\]

Let $p_i\in M_i$ such that via the $\eps$-Hausdorff approximation maps $\varphi_i: M_i\to X'$, we have $\varphi_i(p_i)=p$. Then
\[
d_{GH}((M_i,p_i,r_i^{-1}\eps_ig, \delta r_i^{-1/2}), (U,p,r_i^{-1}g_\infty',\delta r_i^{-1/2}))\leq r_i^{-1}\mu_i.
\]
By the triangle inequality, we have
\[
d_{GH}((M_i,p_i,r_i^{-1}\eps_i g, \delta r_i^{-1/2}), (\mathbb R^q,0,g_E,\delta r_i^{-1/2}))\leq r_i^{-1}\mu_i+\sigma_i.
\]
We can choose $r_i$ going to zero slowly  enough so that
\[
r_i^{-1}\mu_i+\sigma_i \to 0,
\]
and this completes the proof.
\end{proof}

\begin{lem} \label{lemGH2}
Let $(M,g)$  be a complete noncompact Riemannian manifold. Assume that on a sequence of geodesic balls $M_i=B_{x_i}(R_i)$ with $R_i\to \infty$ there exist $\delta_i \to 0$ such that
\[
\mathrm{Ric}_{M_i} \geq - \delta_i.
\]
Then for each $\lambda\geq 0$ there exist smooth approximate eigenfunctions $\phi_i$ whose support lies in $M_i$ such that
\[
\|(\Delta - \lambda) \phi_i\|  \leq o(1) \|\phi_i\|
\]
with $o(1)\to 0$ as $i\to \infty$. As a result the $L^2$ spectrum of the Laplacian on $M$ is $[0,\infty)$.
\end{lem}

\begin{proof}
Replacing $r_i^{-1}\eps_i$ by $\eps_i$ in Lemma \ref{lemGH}, we may  assume that
\[
d_{GH}\left((M_i, x_i, \eps_i g, 5),(\mathbb{R}^q, 0, g_E, 5)\right) \to 0.
\]
By \cite{CCoII}*{Theorem 1.2}, there exists a harmonic map
\[
\Phi_i: (M_i, x_i, \eps_i g, 3) \to \mathbb{R}^q
\]
with  $\Phi_i(M_i, x_i, \eps_i g, 1) \subset (\mathbb{R}^q, 0, g_E, 2)$ and
\[
\mathrm{Lip} \,\Phi_i \leq c(n)
\]
with $c(n)$ independent of $i$. Writing the map $\Phi_i$ in coordinates, $\Phi_i = (b_{i,1}, \cdots, b_{i,q})$, and using  inequality (1.23) of \cite{CCoII} in the proof of the same theorem we obtain
\begin{align*}
&\int_{(M_i, x_i, \eps_i g, 1)} \sum_j |\mathrm{Hess} \,b_{i,j}|^2 \leq o(1) \cdot \mathrm{Vol} ((M_i, x_i, \eps_i g, 1)), \quad and \\
&\int_{(M_i, x_i, \eps_i g, 1)} \sum_j  |\<\n b_{i,j}, \n b_{i,l}\>-\delta_{jl}|\leq o(1)\cdot\mathrm{Vol} ((M_i, x_i, \eps_i g, 1))
\end{align*}
with $o(1)\to 0$ as $i \to \infty$.
We use the diffeomorphism $\sigma_i:(M_i, x_i, g, \eps_i^{-1/2}) \to (M_i, x_i, \eps_i g, 1)$ and a similar map on $(\mathbb{R}^q, 0, g_E, 1)$ to get rescaled  maps
\[
\Phi_i: (M_i, x_i,   g,  \eps_i^{-1/2}) \to  (\mathbb{R}^q, 0, \eps_i^{-1} g_E, \eps_i^{-1/2})=(\mathbb{R}^q, 0,  g_E, \eps_i^{-1/2}),
\]
which are also harmonic. It is clear  that the rescaled maps also satisfy $\mathrm{Lip} \,\Phi_i \leq c(n)$ and
\begin{align} \label{ric_2}
\begin{split}
&\int_{(M_i, x_i, g, {\eps_i}^{-1/2})} \sum_j |\mathrm{Hess}\, b_{i,j}|^2 \leq o(1)\cdot\eps_i\cdot \mathrm{Vol} ((M_i, x_i, g, \eps_i^{-1/2})),\\
&\int_{(M_i, x_i, g, {\eps_i}^{-1/2})} \sum_j  |\<\n b_{i,j}, \n b_{i,l}\>-\delta_{jl}|\leq o(1)\cdot\eps_i\cdot\mathrm{Vol} ((M_i, x_i, g, \eps_i^{-1/2})).
\end{split}
\end{align}

Let $r(x)$ be the distance to the point $0\in \mathbb{R}^q$.  We will denote $ B_i(r)=\{ x\in  M_i \; \mid \; |\Phi_i(x)|\leq r \; \}$ and $\mathrm{Vol} (B_{x_i}(r)) = V_i(r). $

By estimate (1.44) and inequality (1.6) of Theorem 1.2 in \cite{CCoII} we have, after rescaling, that for any $q_1, q_2 \in (M_i, x_i, g, {\eps_i}^{-1/2})$
\[
|\overline{\Phi_i(q_1),\Phi(q_2)} - \overline{q_1,q_2}| \leq o(1) \eps_i^{-1/2}  \]
where $\overline{x,y}$ denotes the distance in the respective metrics. This implies that for any $\frac 12 {\eps_i}^{-1/2}<r<{\eps_i}^{-1/2}$
\[
(M_i, x_i, g, (1 - o(1)) r\,) \subset B_i(  r) \subset (M_i, x_i, g, (1 + o(1))  r\,).
\]
Set $4\rho_i +1 =1/\sqrt{\eps_i}$. Together with the fact that the lower bound of the Ricci curvature of the balls of radius $2/\sqrt{\eps_i}$   approaches 0 and  the metrics are $\eps$-close,  we can  find $i$ large enough such that
\begin{equation} \label{ric_7}
 V_i(3\rho_i) - V_i(2\rho_i) \geq c_1(n)\,  V_i(4\rho_i+1)
\end{equation}
for some $c_1(n)>0$.

For a fixed $\lambda\geq 0$ we let $\phi_o(r)=e^{i \sqrt{\lambda}\, r} \chi(r),$  where $\chi(r)$ is a smooth function with support in $(\mathbb{R}^q, 0, g_E, {\eps_i^{-1/2}})$.  We choose $\chi$ such that $0\leq \chi(r) \leq 1$, $\chi(r)=1$ for $r\in[2\rho_i,3\rho_i]$, $\chi(r)=0$ for $r\leq \rho_i, r\geq 4\rho_i$ and satisfying $|\chi'(r)|, |\chi''(r)| \leq C\eps_i^{1/2}$.

Define the test functions
\[
\phi_i= \phi_o \circ \Phi_i
\]
which are now supported in $(M_i, x_i, g, \eps_i^{-1/2})$. We observe that in the coordinates $(x_1,\cdots,x_q)$ of $\mathbb R^q$
\begin{equation}
\begin{split}
\Delta \phi_i & = -\sum_{j,l=1}^q\frac{\p^2\phi_o}{\p x_j \, \p x_l}  \<\n b_{i,j}, \n b_{i,l}\> \\
&  = -\sum_{j\neq l} \frac{\p^2\phi_o}{\p x_j \, \p x_l}  \; \<\n b_{i,j}, \n b_{i,l} \>   + \Delta_E  \phi_o  -
\sum_{j=1}^q \frac{\pa^2\phi_o}{\pa x_j^2}(|\n b_{i,j}|^2-1)
\end{split}
\end{equation}
since the $b_{i,j}$ are harmonic, where $\Delta_E$ is the Laplacian on the Euclidean space and $\n_k $ denotes the covariant derivative in the $k$th direction over $M_i$.

Then
\begin{equation*}
\begin{split}
|(\Delta - \lambda) \phi_i| & \leq |\mathrm{Hess}_E(\phi_o)| \; |\sum_{j\neq l} \<\n b_{i,j}, \n b_{i,l} \>| + |(\Delta_E - \lambda) \phi_o| \\
& \quad + \sum_j (|\mathrm{Hess}_E(\phi_o)| \cdot||\n b_{i,j}|^2 -1|).
\end{split}
\end{equation*}
It is well known that
\[
\|(\Delta_E - \lambda) \phi_o\|\leq o(1) V(4\rho_i +1).
\]
Using the fact that the Hessian of $\phi_o$ over the Euclidean space is bounded, that the $\n \Phi_i$ are uniformly bounded and that $\phi_o$ is an approximate eigenfunction on $\mathbb R^q$, together with estimate~\eqref{ric_2} and the above inequality, we have
\begin{equation} \label{ric_8}
\begin{split}
\|(\Delta - \lambda) \phi_i\|^2 & \leq C \; \int_{B_{x_i}(4\rho_i+1)}\sum_{j\neq l} |\<\n b_{i,j}, \n b_{i,l} \>|^2    + \|(\Delta_E - \lambda) \phi_o\|^2 \\
& \quad + C \sum_j \int_{B_{x_i}(4\rho_i+1)}   | |\n b_{i,j}|^2 -1|^2\\
&  \leq o(1) V(4\rho_i +1).
\end{split}
\end{equation}
On the other hand,
\[
\|\phi_i\|^2    \geq  V_i(3\rho_i) - V_i(2\rho_i) \geq c_1(n) V_i(4\rho_i +1)
\]
by~\eqref{ric_7}.
Combining the above with \eref{ric_8}, we obtain the desired estimate
\[
\|(\Delta - \lambda) \phi_i\|  \leq o(1) \|\phi_i\|.
\]
\end{proof}

Lemma \ref{lemGH2} allows us to find $L^2$ approximate eigenfunctions for the $[0,\infty)$ spectrum of manifolds with Ricci curvature asymptotically nonnegative. This was not previously possible, even on manifolds with Ricci curvature nonnegative. It also provides a novel proof of the same result due to the second author and D. Zhou~\cite{Lu-Zhou_2011} without resorting to Sturm's $L^p$ independence result \cite{sturm}.

Based on the above result, we are now able to prove Theorem~\ref{thmRic}.
\begin{proof}[Proof of Theorem~\ref{thmRic}]
Consider the forms $\omega_o=\phi_o(r) dx_1 \wedge \cdots \wedge dx_{k}$ on $\mathbb{R}^q$, with $\phi_o(r)$ defined as in the previous Lemma.  These are $L^2$ approximate $k$-eigenforms for $\lambda$ on $\mathbb{R}^q$.

Define the test $k$-forms
\[
\omega_i = \Phi_i^*(\omega_o)= \phi_i  \; db_{i,1} \wedge \cdots \wedge db_{i,k}
\]
for $\phi_i$ defined as in Lemma~\ref{lemGH2}.

By the Weitzenb\"ock formula, it is well known that
\begin{equation} \label{ric_1}
\begin{split}
(\Delta  - \lambda) \omega_i = & (\Delta \phi_i - \lambda \phi_i \,)  \, db_{i,1} \wedge \cdots \wedge db_{i,k} -2 \n_{\n \phi_i } (db_{i,1} \wedge \cdots \wedge db_{i,k}) \\
& + \phi_i \Delta  (db_{i,1} \wedge \cdots \wedge db_{i,k}).
\end{split}
\end{equation}

We will apply Corollary \ref{cor21} to show that $\lambda$ belongs to the spectrum of $M$.  Using formula \eref{ric_1}  together with the properties of the $b_{i,l}$ we get
\begin{equation} \label{ric_3}
\begin{split}
|(\,(\Delta +\alpha)^{-m}\omega_i,(\Delta  - \lambda) \omega_i \,)|  & \leq C \,\| \omega_i\|  \cdot  \left[
\|(\Delta - \lambda) \phi_i   \|   + \sum_j  \| |\n \phi_i|\cdot |\mathrm{Hess} (b_{i,j})| \, \|   \right]\\
& \quad + |(\,(\Delta +\alpha)^{-m}\omega_i,\phi_i \Delta (db_{i,1} \wedge \cdots \wedge db_{i,k}) \,)|,
\end{split}
\end{equation}
where we have used the fact that $(\Delta +\alpha)^{-m}$ is bounded on $L^2$ for the first two terms in the right side. We want to show that the right side is bounded above by $o(1) \|\omega_i\|^2$.

By the following algebraic inequality (using the fact that the gradients of the $b_{i,j}$ are uniformly bounded), we have
\[
|db_{i,1} \wedge \cdots \wedge db_{i,k}|^2\geq 1-C(n)\sum_{j,l}|\<\n b_{i,j}, \n b_{i,l}\>-\delta_{jl}|.
\]
Thus by~\eqref{ric_2} and ~\eqref{ric_7}, for $i$ sufficiently large we obtain
\[
\begin{split}
\|\omega_i\|^2 &\geq \int_{B_{x_i}(3\rho_i)\backslash B_{x_i}(2\rho_i)} |db_{i,1} \wedge \cdots \wedge db_{i,k}|^2\\
&\geq V(3\rho_i)-V(2\rho_i)- o(1) V(4\rho_i+1)\geq \frac 12c_1(n)V(4\rho_i +1).
\end{split}
\]

Combining the above with \eref{ric_8} we have
\begin{equation} \label{ric_4}
\|(\Delta - \lambda) \phi_i\|  \leq o(1) \|\omega_i\|
\end{equation}

Similarly, using \eref{ric_2} we can control the second term of \eref{ric_3}
\begin{equation} \label{ric_5}
\begin{split}
\sum_j  \| |\n \phi_i|\cdot |\mathrm{Hess} (b_{i,j})| \, \|   \leq o(1) \; V(4 \rho_i+1)^{1/2} \leq  o(1) \|\omega_i\|.
\end{split}
\end{equation}

Finally, for the third term, we let $\eta_i = (\Delta +\alpha)^{-m}\omega_i$ and observe that
\begin{equation*}
|(\,\eta_i ,\phi_i \Delta (db_{i,1} \wedge \cdots \wedge db_{i,k}) \,)| = |(\, \delta(\phi_i \eta_i) , \delta(db_{i,1} \wedge \cdots \wedge db_{i,k}) \,)|.
\end{equation*}
Since $ \delta(\phi_i \eta_i) = -\iota(\n \phi_i) \eta_i + \phi_i \delta \eta_i $, and
\begin{equation*}
\begin{split}
\|  \delta \eta_i \|^2 & \leq \|  \delta \eta_i \|^2 + \|  d \eta_i \|^2  = (\eta_i, \Delta (\Delta +\alpha)^{-m}\omega_i) \\
& =  (\eta_i,  (\Delta +\alpha)^{-m+1}\omega_i) - \alpha (\eta_i,  (\Delta +\alpha)^{-m}\omega_i ) \leq C\| \omega_i\|^2,
\end{split}
\end{equation*}
we have
\[
\|\delta (\phi_i\eta_i) \|\leq C\|\omega_i\|.
\]
Therefore,
\begin{align}  \label{ric_6}
\begin{split}
&|(\, \delta(\phi_i \eta_i) , \delta(db_{i,1} \wedge \cdots \wedge db_{i,k}) \,)| \leq C \|\omega_i\|  \;  \sum_j  \|\mathrm{Hess} (b_{i,j})\|  \leq o(1)   \|\omega_i\|\cdot V(4\rho_i+1).\\
\end{split}
\end{align}

Using \eref{ric_4}, \eref{ric_5} and \eref{ric_6} to estimate the right side of \eref{ric_3} we get
\[
|(\,(\Delta +\alpha)^{-m}\omega_i,(\Delta  - \lambda) \omega_i \,)|  \leq o(1)  \,\| \omega_i\|^2.
\]
The theorem follows from Corollary \ref{cor21} after rescaling each $\omega_i$ by its $L^2$ norm.
\end{proof}

We would like to remark that since $q>0$, the result of Theorem \ref{thmRic} is true over a manifold with asymptotically nonnegative Ricci curvature for the case of functions and 1-forms and is consistent with Theorem \ref{thm1forms}.  It is therefore also consistent with Corollary \ref{cor32}. For higher order forms however, it has not been possible to compute the $k$-form spectrum of manifolds with asymptotically nonnegative Ricci curvature. Our result addresses this case, at least partially.

If $q\geq n/2$, then for any $0\leq k\leq n$, the spectrum on $k$-forms is $[0,\infty)$. This observation gives us the following result
\begin{corl}
Let $M$ be a complete noncompact Riemannian manifold with nonnegative Ricci curvature.
Assume that  $M$ has Euclidean volume growth.  Then
\[
\sigma_{\rm ess}(k,\Delta, M) = [0,\infty)
\]
for $0\leq k\leq n$.
\end{corl}

Finally,  we would like to make the following  observation which is not directly related to
our results, but explains why we are able to obtain spectral information for the Hodge Laplacian in the absence of bounds for the curvature tensor. We assume that the Ricci curvature of $M$ is small.

We shall suppress the subscript $i$ and to use  $b_1, \cdots, b_q$ to denote the harmonic functions $b_{i,1}, \cdots, b_{i,q}$. By a straightforward computation, we have
\begin{equation*}
\begin{split}
\Delta  (db_{1} \wedge \cdots &\wedge db_{k}) =  \sum_l  db_{1} \wedge \cdots  \wedge  \Delta (db_{j}) \wedge \cdots \wedge db_{k}\\
& - 2 \sum_{i=1}^n  \sum_{l, m}   db_{1} \wedge \cdots  \wedge \n_{{\p}_i} (db_l) \wedge \cdots \wedge \n_{{\p}_i} (db_m) \wedge \cdots \wedge db_{k}
\end{split}
\end{equation*}
pointwise, where $\{{\p}_i\}$ is a normal coordinate frame field  with covectors $\{dx^i\}$.  Since the $b_j$ are harmonic, we get
\[
\Delta  (db_{1} \wedge \cdots \wedge db_{k}) =-2 \sum_{i=1}^n  \sum_{l, m}   db_{1} \wedge \cdots  \wedge \n_{{\p}_i} (db_l) \wedge \cdots \wedge \n_{{\p}_i} (db_m) \wedge \cdots \wedge db_{k}.
\]
By ~\eqref{ric_2}, we conclude that the average of the $L^1$ norm of $\Delta  (db_{1} \wedge \cdots \wedge db_{k})$ is very small without any further assumptions on the curvature tensor of $M$.

Similarly,
\begin{equation*}
\begin{split}
\n^* \n(db_{1} \wedge \cdots &\wedge db_{k}) =  \sum_l  db_{1} \wedge \cdots  \wedge  \n^* \n (db_{j}) \wedge \cdots \wedge db_{k}\\
& - 2 \sum_{i=1}^n  \sum_{l, m}   db_{1} \wedge \cdots  \wedge \n_{{\p}_i} (db_l) \wedge \cdots \wedge \n_{{\p}_i} (db_m) \wedge \cdots \wedge db_{k},
\end{split}
\end{equation*}
where $\n^* \n=-{\rm Tr}(\n^2)$ is the covariant Laplacian. By the Weitzenb\"ock formula,
\[
\n^* \n(db_j)=\Delta (db_j) - \iota(\n b_j){\rm Ric}
\]
since the Weitzenb\"ock tensor on 1-forms coinsides with the Ricci tensor. Thus
$\n^* \n (db_{1} \wedge \cdots \wedge db_{k})$ is also small in $L^1$.  As a result,
\[
\mathcal{W}_k(db_{1} \wedge \cdots \wedge db_{k})= \Delta (db_{1} \wedge \cdots \wedge db_{k})- \n^* \n (db_{1} \wedge \cdots \wedge db_{k})
\]
is small in the $L^1$ norm under the assumption of small Ricci curvature. Recall that
\[
\mathcal{W}_k=-\sum_{i,j} \omega^i \wedge \iota(V_j) R_{V_i V_j}
\]
where $\{V_i\}$ is an orthonormal frame field with dual co-frame $\{\omega^j\}$ and  $R$ is the curvature tensor. Here $R_{XY} = D_X D_Y -D_Y D_X -D_{[X,Y]}$.


\end{document}